\documentclass[a4paper,11pt]{article}
\usepackage{amsmath, amsfonts, amscd, amssymb, amsthm, enumerate, xypic}

\def\bfB{\mathbf{B}}

\DeclareMathOperator{\id}{\operatorname{id}}

\DeclareMathOperator{\Mat}{\operatorname{M}}
\DeclareMathOperator{\Matt}{\operatorname{T}}
\DeclareMathOperator{\Hom}{\operatorname{Hom}}

\DeclareMathOperator{\Ker}{\operatorname{Ker}}

\DeclareMathOperator{\End}{\operatorname{End}}

\DeclareMathOperator{\GL}{\operatorname{GL}}
\DeclareMathOperator{\Vect}{\operatorname{span}}
\DeclareMathOperator{\im}{\operatorname{Im}}
\DeclareMathOperator{\tr}{\operatorname{tr}}

\DeclareMathOperator{\car}{\operatorname{char}}

\DeclareMathOperator{\rk}{\operatorname{rk}}

\renewcommand{\setminus}{\smallsetminus}
\renewcommand{\epsilon}{\varepsilon}

\def\F{\mathbb{F}}
\def\R{\mathbb{R}}

\def\N{\mathbb{N}}

\def\calA{\mathcal{A}}
\def\calB{\mathcal{B}}

\def\calM{\mathcal{M}}
\def\calN{\mathcal{N}}

\def\calP{\mathcal{P}}

\def\calT{\mathcal{T}}
\def\calU{\mathcal{U}}
\def\calV{\mathcal{V}}

\def\calX{\mathcal{X}}

\def\calZ{\mathcal{Z}}

\def\lcro{\mathopen{[\![}}
\def\rcro{\mathclose{]\!]}}

\theoremstyle{definition}
\newtheorem{Def}{Definition}[section]

\theoremstyle{plain}
\newtheorem{theo}{Theorem}[section]

\newtheorem{lemma}[theo]{Lemma}
\newtheorem{claim}{Claim}[section]

\theoremstyle{plain}

\theoremstyle{remark}

\title{Spaces of triangularizable matrices (III): Perfect non-quadratically closed fields with characteristic $2$}
\author{Cl\'ement de Seguins Pazzis\footnote{Universit\'e de Versailles Saint-Quentin-en-Yvelines, Laboratoire de Math\'ematiques
de Versailles, 45 avenue des Etats-Unis, 78035 Versailles cedex, France}
\footnote{e-mail address: clement.de-seguins-pazzis@ac-versailles.fr}}

\begin{document}

\thispagestyle{plain}

\maketitle
\begin{abstract}
Given a field $\F$ and an integer $n \geq 2$,
denote by $t_n(\F)$ the greatest possible dimension for a vector space of $n$-by-$n$ matrices over $\F$
in which every element is triangularizable. It was recently proved that $t_n(\F)=\frac{n(n+1)}{2}$ if and only if $\F$ is not quadratically closed,
with the possible exception of finite fields with characteristic $2$ and less than $n-1$ elements.

In this article, we prove that the equality $t_n(\F)=\frac{n(n+1)}{2}$ holds for all perfect non-quadratically closed fields with characteristic $2$ -- with the exception of fields with cardinality $2$ -- and for these fields we obtain a key result for a future analysis of the spaces that have the critical dimension $t_n(\F)$.
\end{abstract}




\section{Introduction}

Let $\F$ be a field. For nonnegative integers $n$ and $p$, denote by $\Mat_{n,p}(\F)$ the vector space of all $n$-by-$p$ matrices with entries in $\F$,
by $\Mat_n(\F)$ the algebra of all $n$-by-$n$ matrices with entries in $\F$, by $\Matt_n(\F)$
its subalgebra of all upper-triangular matrices, and by $\mathfrak{sl}_n(\F)$ its linear subspace of all trace zero elements.

A subset $\calX$ of $\Mat_n(\F)$ is called \textbf{weakly triangularizable} when every element of $\calX$ is triangularizable over $\F$,
i.e., conjugated to an element of $\Matt_n(\F)$, or equivalently when every element of $\calX$ is annihilated by a split polynomial with coefficients in $\F$.
We adopt a similar definition for subsets of $\End(V)$ when $V$ is a finite-dimensional vector space over $\F$.
Given an integer $n \in \N$, we denote by $t_n(\F)$ the greatest possible dimension for a weakly triangularizable linear subspace of
$\Mat_n(\F)$ (or, for a weakly triangularizable subspace of $\End(V)$, where $V$ is an $n$-dimensional vector space over $\F$).
Obviously $t_n(\F) \geq \frac{n(n+1)}{2}$ because $\Matt_n(\F)$ is a weakly triangularizable subspace of $\Mat_n(\F)$.

The study of weakly triangularizable matrix spaces has been initiated recently \cite{dSPtriangularizable}, in part due to its connection with the classical Gerstenhaber theorem on spaces
of nilpotent matrices \cite{Gerstenhaber}. Remembering that a matrix over a field is nilpotent if and only if it is similar to a strictly upper-triangular matrix,
weakly triangularizable spaces can be viewed as a natural variation of Gerstenhaber's theorem.
Strikingly, interesting results on weakly triangularizable matrix spaces have emerged only very recently:
a potential reason might be that the value of $t_n(\F)$ profoundly depends on the arithmetic properties of the field $\F$ under consideration.
For example, if $\F$ is algebraically closed then $t_n(\F)=n^2$ for every integer $n \geq 1$, whereas if $\F$ is the field of real numbers it is easy to see that
$t_n(\F)=\frac{n(n+1)}{2}$ by noting that if a skew-symmetric real matrix is triangularizable over the reals then it is zero
(indeed, it is classical that every skew-symmetric real matrix is diagonalisable over the complex numbers, with all eigenvalues pure imaginary).

Let us now recall the main result from \cite{dSPtriangularizable}:

\begin{theo}\label{theo:classical}
Let $\F$ be a field, and $n \geq 2$ be an integer. Assume that either $\car(\F) \neq 2$ or $|\F| \geq n-1$.
Then $t_n(\F)=\frac{n(n+1)}{2}$ if and only if $\F$ is not quadratically closed.
\end{theo}

Here, we must stress that we use the definition of quadratic closeness from Galois Theory: a field $\F$ is quadratically closed
when it has no algebraic extension of degree $2$, i.e., all polynomials with degree $2$ and coefficients in $\F$ have a root in $\F$.
This is different from a popular (yet misguided, in our opinion, due to the characteristic $2$ issue) way of defining quadratically closed fields
as those for which $x \in \F \mapsto x^2 \in \F$ is surjective. For example, under our definition no finite field with characteristic $2$ is quadratically
closed, whereas all finite fields are perfect. We will call a field $\F$ \textbf{NRC} (for ``Non-Root-Closed") when the mapping $x \in \F \mapsto x^2 \in \F$ is nonsurjective.

It should be also noted that in \cite{dSPtriangularizable}, the stated provision for the above theorem was that $\car(\F) \neq 2$ or $\F$
is infinite. However, a close examination of the arguments for fields with characteristic $2$, which essentially rest upon a
theorem of Atkinson and Lloyd on primitive spaces of matrices with bounded rank \cite{AtkinsonLloydPrim}, shows that it holds under the milder assumption that $|\F| \geq n-1$.

We also mention that \cite{dSPtriangularizable} was not only concerned with the determination of $t_n(\F)$, but also with the structure of the spaces
that reach the maximal dimension $t_n(\F)$, still in the case of a field that is not quadratically closed. This determination was achieved
under the assumption that $\F$ is non-quadratically closed and $|\F| \geq n$ (again, it is formally stated there that the field should be infinite,
but a close examination of the proof reveals that the looser property $|\F| \geq n$ is sufficient).
In \cite{dSPtriangularizable2}, the classification was generalized to all finite fields with odd characteristic.

Hence, there remains two challenging problems in the study of weakly triangularizable spaces:
\begin{enumerate}[(1)]
\item Extend the equality $t_n(\F)=\frac{n(n+1)}{2}$ to all finite fields with characteristic $2$;
\item Classify, in the same setting, the weakly triangularizable subspaces that have dimension $t_n(\F)$, which we call the
\textbf{optimal} weakly triangularizable subspaces.
\end{enumerate}

The present work deals with problem (1), and we will solve it with the exception of fields with cardinality $2$
(for which our techniques entirely fail). We will also pave the way for a solution of problem (2) for the same fields.

Before moving forward, it is important to explain what has been proved in \cite{dSPtriangularizable} regarding fields with characteristic $2$,
in particular for optimal spaces. To this end, we need the notation for the \emph{joint} of two (or several) spaces of square matrices.
Let $n$ and $p$ be nonnegative integers and $\calA$ and $\calB$ be respective subsets of $\Mat_n(\F)$ and $\Mat_p(\F)$.
Their \textbf{joint}, denoted by $\calA \vee \calB$, is the set of all matrices of $\Mat_{n+p}(\F)$ of the form
$$\begin{bmatrix}
A & C \\
[0]_{p \times n} & B
\end{bmatrix}$$
where $A \in \calA$, $B \in \calB$ and $C \in \calM_{n,p}(\F)$.
The operator $\vee$ is clearly associative, and we naturally extend it as an $n$-ary operator on lists of subsets of square matrices.
We can now state one of the main results of \cite{dSPtriangularizable}:

\begin{theo}\label{theo:grandscardinaux}
Let $\F$ be a perfect field of characteristic $2$ that is not quadratically closed.
Let $n \geq 1$ be an integer such that $|\F| \geq n$, let $V$ be an $n$-dimensional vector space over $\F$,
and let $S$ be a weakly triangularizable subspace of $\End(V)$ with dimension $\dbinom{n+1}{2}$.
Then, in some basis of $V$, the space $S$ is a represented by a joint of matrix spaces among $\Mat_1(\F_2)$ and $\mathfrak{sl}_2(\F)$.
\end{theo}

In \cite{dSPtriangularizable}, this result was stated with the assumption that $\F$ be infinite (instead of $|\F| \geq n$), but an examination
of the proof (which essentially relies upon a reduction to Atkinson's theorem on primitive spaces of bounded rank matrices) reveals that the assumption
$|\F| \geq n$ suffices. This observation will turn out to be critical for our own proof.

Now, here is the main result we are aiming at in this article:

\begin{theo}\label{theo:majodim}
Let $\F$ be a perfect field with characteristic $2$ that is not quadratically closed, and with $|\F| > 2$.
Then the greatest possible dimension for a weakly triangularizable subspace of $\Mat_n(\F)$ is $\frac{n(n+1)}{2}\cdot$
\end{theo}

Our proof will not require any finiteness assumption on $\F$, although a small step (Lemma \ref{lemma:polylemma})
has a simpler proof for finite fields.

This result is not the only one that we will prove in this article. Indeed, we will also
obtain a result on the existence of so-called \emph{adapted vectors}. Adapted vectors are the key
to decipher the spaces that have the greatest possible dimension when the field $\F$ has small cardinality: it is precisely the method that was used
in \cite{dSPfeweigenvalues2} to study those spaces over finite fields with odd characteristic.
This refined result will help us, in a future sequel to this article, to generalize Theorem \ref{theo:grandscardinaux} to all finite fields
with characteristic $2$ and more than $2$ elements.

We now quickly explain the main strategy for proving Theorem \ref{theo:majodim}.
It was discovered in \cite{dSPtriangularizable} that proving the analogue theorem for NRC fields could be done in a way that is very close
to the one used a decade ago \cite{dSPfeweigenvalues} to obtain similar results on spaces of matrices that have at most two eigenvalues in their field of definition
(a special case of the strong Loewy-Radwan problem). This connection is indeed quite strange, because although both problems deal with the spectrum of the matrices under consideration, they are almost polar opposites (in one problem, we assume that there are not too many distinct eigenvalues, while in the other one we assume on the contrary that there be as many eigenvalues, counted with multiplicities, as possible!).
The present article will exploit recent advances in the problem of spaces of matrices with no more than two eigenvalues
\cite{dSPfeweigenvalues2}. We will recycle some parts of the strategy, with the main difficulties entirely concentrated in adapting some key lemmas.

The remainder of the article is organized as follows.
The idea is to arrive as quickly as possible to the proof of Theorem \ref{theo:majodim}.
\begin{itemize}
\item In Section \ref{section:basicmaterial}, we introduce the main concepts for the proof of Theorem \ref{theo:majodim}.
We briefly recall the definition of an adapted vector (for a subspace of endomorphisms) and explain how such a vector can be used for an inductive proof of Theorem \ref{theo:majodim}. Then we recall some key notions in combinatorics of
finite-dimensional vector spaces over finite fields, along with an important property of the null set of homogeneous polynomial functions on such spaces
(the Vanishing Lemma). Finally we recall some facts on trace dual orthogonal complements of operator spaces.

\item The second part of the article deals with the existence of adapted vectors in weakly triangularizable subspaces.
We shall in fact start with the weaker notion of a weakly adapted vector, which turns out to be sufficient to yield
the upper-bound in Theorem \ref{theo:majodim}. The proof of the existence of such vectors will be obtained by induction,
requiring the use of a difficult result from \cite{dSPtriangularizable} for the very special case of $3$-by-$3$ matrices over $\F_4$
(the remaining part of the proof is, however, essentially self-contained).

\item The next, and largest, part of the article is devoted to the issue of the existence of adapted vectors.
It is not true that such vectors always exist, and the obstruction comes from certain operator spaces that we call \emph{hurdles}:
these spaces are defined in Section \ref{section:hurdles}. An important result proved in \cite{dSPtriangularizable} (called the Second Erasure Lemma there)
essentially says that understanding a weakly triangularizable subspace of $\Mat_n(\F)$ that is a hurdle can be reduced
to the classification of the weakly triangularizable subspaces of $\Mat_{n-2}(\F)$.
The main result of the last part of the article is that every weakly triangularizable subspace that is not a hurdle
has an adapted vector. This involves a very intricate inductive proof, which mimics the analogue one
for spaces of matrices with at most two eigenvalues \cite{dSPfeweigenvalues2}, but with several different technical details.
These results are essential to pave the way for a future study of the optimal weakly triangularizable spaces over finite fields with characteristic $2$.
\end{itemize}

Before we move forward, we would like to comment on the fact that we left out fields with cardinality $2$.
We conjecture indeed that the result of Theorem \ref{theo:majodim} holds for such fields as well.
Although we want to reach the greatest possible generality, fields with cardinality $2$ have shown remarkable resistance to our methods of analysis.
Our proof here will use the assumption $|\F|>2$ in countless many instances, and we have failed to adapt it to encompass $\F_2$.
It is clear that a completely different approach will be required to handle $\F_2$.
Interestingly, for $\F_2$ the weakly triangularizable spaces are the $\overline{1}^\star$-spec subspaces studied in \cite{dSPfeweigenvalues2}, i.e.,
the linear subspaces of matrices in which every matrix has at most one nonzero eigenvalue in an algebraic closure of $\F$. However, this remark is of no help to us
because the methods from \cite{dSPfeweigenvalues2} do not have any clear adaptation to fields with cardinality $2$.

\section{Main technical concepts, and the adapted vectors method}\label{section:basicmaterial}

\subsection{Basics on duality}\label{section:duality}

Throughout, we let $V$ be a nonzero finite-dimensional vector space, and we denote by $V^\star:=\Hom(V,\F)$ its dual vector space.
For a nonzero vector $x \in V$ we denote by $x^\bot$ the space of all linear forms $\varphi \in V^\star$ such that $\varphi(x)=0$
(it is a linear hyperplane of $V^\star$).
For a vector $y \in V$ and a linear form $\varphi \in V^\star$, we write
$$\varphi \otimes y : x \in V \mapsto \varphi(x)\,y.$$
Beware that most authors use the notation $y \otimes \varphi$ instead, but our notation is better suited to computations, e.g, as
we observe that $\tr(\varphi \otimes y)=\varphi(y)$.
Of course, here we have just described all the endomorphisms of $V$ with rank at most $1$.

For a linear subspace $H$ of $V$, we denote by
$$H^o:=\{f \in V^\star : \forall x \in H, \; f(x)=0\}$$
its dual orthogonal, which is a linear subspace of $V^\star$ with dimension $\dim V-\dim H$;
for a linear subspace $G$ of $V^\star$, we denote by
$${}^o G:=\{x \in V : \forall f \in G, \; f(x)=0\}$$
its pre-dual orthogonal,  which is a linear subspace of $V$ with dimension $\dim V-\dim G$.
Classically ${}^o (H^o)=H$ and $({}^o G)^o=G$.

\subsection{Definition of adapted vectors, and the basic inductive method}\label{section:adaptedvectorsinduction}

\begin{Def}
Let $S \subseteq \End(V)$ be a linear subspace of $\End(V)$.

A nonzero vector $x \in V \setminus \{0\}$ is called
\textbf{$S$-adapted} whenever $S$ contains no operator with range $\F x$ and trace zero, or in other words
$S \cap (x^\bot \otimes x)=\{0\}$.

A nonzero vector $x \in V \setminus \{0\}$ is called \textbf{weakly $S$-adapted} whenever
$\dim(S \cap (V^\star \otimes x)) \leq 1$.
\end{Def}

Note that the definition of a weakly $S$-adapted vector is different here from the one chosen in \cite{dSPfeweigenvalues2},
where instead we required the weaker assumption that $\dim(S \cap (x^\bot \otimes x)) \leq 1$.

Note that every $S$-adapted vector is also weakly $S$-adapted. In \cite{dSPtriangularizable} we tried to obtain $S$-adapted vectors
to prove the inequality $t_n(\F) \leq \frac{n(n+1)}{2}$ by induction, but it is easily seen that having weakly $S$-adapted vectors is full sufficient for this task.
Let us immediately explain why.

So, let $V$ be an $n$-dimensional vector space over $\F$, with $n \geq 2$, and let $S$ be a weakly triangularizable subspace of $\End(V)$.
Let $x \in V \setminus \{0\}$, which we take as completely arbitrary at first. We consider the linear subspace
$$S_x :=\{u \in S : \; u(x) \in \F x\},$$
and we apply the rank theorem to obtain
$$\dim S \leq n-1+\dim S_x.$$
Then, every $u \in S_x$ induces an endomorphism $\overline{u}$ of $V/\F x$, and we set
$$\overline{S_x} :=\{\overline{u} \mid u \in S_x\} \subseteq \End(V/\F x),$$
which is a weakly triangularizable subspace of $\End(V/\F x)$. The rank theorem yields
$$\dim S_x =\dim (S \cap (V^\star \otimes x))+\dim \overline{S_x}.$$
Hence, if $x$ is weakly $S$-adapted and we have $\dim \overline{S_x} \leq \dbinom{n}{2}$ thanks to an induction hypothesis,
then we immediately recover $\dim S \leq \dbinom{n+1}{2}$.

So, in order to obtain Theorem \ref{theo:majodim} it suffices to prove that \emph{every} weakly triangularizable subspace of $\End(V)$
has a weakly adapted vector (whatever the space $V$, with nonzero dimension, under consideration).
However, this is easier said than done.

\subsection{The Covering Lemma, and an important lemma on homogeneous polynomials}

Now we recall two very important tools in our search for adapted vectors or for weakly adapted vectors.
The second one is a generalization of the first one. We refer to \cite{dSPfeweigenvalues} (lemma 2.5 there)
for a proof of the first one, and to \cite{dSPfeweigenvalues2} for a proof of the second one (lemma 4.2 there).

\begin{lemma}[Covering Lemma]\label{coveringlemma2}
Let $p$ be a positive integer such that $|\F|>p$.
Let $V$ be an $n$-dimensional vector space over $\F$, and $(V_i)_{i \in I}$
be a family of linear subspaces of $V$ in which :
\begin{enumerate}[(i)]
\item $|I|=(n-1)p+1$ ;
\item For all $k \in \lcro 1,n-2\rcro$, exactly $p$ vector spaces in the family have dimension $k$;
\item Exactly $p+1$ vector spaces in the family have dimension $n-1$.
\end{enumerate}
Then $V \neq \underset{i \in I}{\bigcup} V_i$.
\end{lemma}

\begin{lemma}[Vanishing Lemma for Homogeneous Polynomials]\label{lemma:vanishinghomogeneous}
Let $d$ and $m$ be positive integers such that $|\F| \geq m \geq d$.
Let $V$ be a vector space with dimension $n$ over $\F$, and
$p : V \rightarrow \F$ be a $d$-homogeneous polynomial function.

Assume that there exists a finite family $(V_i)_{i \in I}$ of proper linear subspaces of $V$ in which:
\begin{enumerate}[(i)]
\item For each $j \in \lcro 1,n-2\rcro$, at most $m-1$ indices $i$ satisfy
$\dim V_i=j$;
\item At most $m-d$ indices $i$ satisfy $\dim V_i=n-1$;
\item The function $p$ vanishes outside of $\underset{i \in I}{\bigcup} V_i$.
\end{enumerate}
Then $p=0$.
\end{lemma}

As was noted in \cite{dSPfeweigenvalues2}, the Covering Lemma can be seen as a special case of Lemma \ref{lemma:vanishinghomogeneous},
but it is useful to state it independently because it is frequently used.

We finish with the notion of a $k$-complex of a vector space:

\begin{Def}
Let $k \geq 2$ and $V$ be an $n$-dimensional vector space.
A \textbf{$k$-complex} of $V$ is a $(k-1)n$-list $(V_1,\dots,V_{(k-1)n})$ of linear subspaces
of $V$ such that $\dim V_i=1+\lfloor \frac{i-1}{k}\rfloor$ for all $i \in \lcro 1,(k-1)n\rcro$,
i.e., the first $k$ spaces have dimension $1$, the next $k$ spaces have dimension $2$, and so on.
\end{Def}

\subsection{Trace orthogonality}

Let $S$ be a linear subspace of $\End(V)$. We set
$$S^\bot:=\{v \in \End(V) : \tr(vu)=0\},$$
which is a linear subspace of $\End(V)$. It is the orthogonal complement of $S$ under the nondegenerate symmetric bilinear form
$(u,v) \in \End(V)^2 \mapsto \tr(uv)$, and in particular it enjoys the property that $\dim S^\bot=\dim \End(V)-\dim S$
and $S^{\bot\bot}=S$.
One important property is the identity
\begin{equation}\label{eq:dualtransitive}
\forall x \in V \setminus \{0\}, \quad \dim (S^\bot x)=\dim V-\dim (S \cap (V^\star \otimes x)),
\end{equation}
for which we refer to lemma 3.3 of \cite{dSPtriangularizable}.

\section{The greatest possible dimension for a weakly triangularizable subspace}

The goal of the present section is to prove the following result, which reinforces Theorem \ref{theo:majodim}:

\begin{theo}\label{theo:weaklyadapted}
Assume that $|\F|>3$ and that $\F$ is not quadratically closed.
Let $V$ be an $n$-dimensional vector space over $\F$, and $S$ be a weakly triangularizable linear subspace of $\End(V)$.
\begin{enumerate}[(a)]
\item If $n \geq 3$ then the union of a $3$-complex of $V$ never contains all the weakly-$S$-adapted vectors.
\item One has $\dim S \leq \dbinom{n+1}{2}\cdot$
\end{enumerate}
\end{theo}

The provision $n \geq 3$ in point (a) is unavoidable because every $3$-complex of $V$ contains $V$ as fourth component if $n=2$.

Note that the result holds even over fields with characteristic other than $2$, and the proof will encompass all such fields with the notable exception of those with cardinality $3$.
As a consequence, it almost fully encompasses theorem 1.1 of \cite{dSPtriangularizable}, leaving out only the fields with cardinality $3$
(for which the approach used in \cite{dSPtriangularizable}, which is unsuited to perfect fields with cardinality $2$, seems unavoidable).

Our proof of Theorem \ref{theo:weaklyadapted} is inductive.
One of its main features is that, in contrast with the method used in \cite{dSPtriangularizable},
we do not separate the proof of point (a) from the one of point (b) in the induction.

Throughout, we let $\F$ be a non-quadratically closed field with more than $3$ elements.

\subsection{The case $n=2$}

Here we quickly settle the case $n=2$. So, we let $V$ be a $2$-dimensional vector space over $\F$, and $S$ be a weakly triangularizable subspace of $\End(V)$.
Since $\F$ is not quadratically closed, there is a monic irreducible polynomial $p$ with degree $2$ in $\F[t]$, and there exists an endomorphism of $V$ with characteristic polynomial $p$, which is then non-triangularizable.
Hence $S \neq \End(V)$ and we conclude that $\dim S \leq \dbinom{3}{2}$.

\subsection{Setting the inductive step up, and a first rough upper-bound for the dimension}\label{section:inductionweaklysetup}

Now, we let $n \geq 3$, we let $V$ be an $n$-dimensional vector space and $S$ be a weakly triangularizable subspace of $\End(V)$.
We assume of course that the statements in Theorem \ref{theo:weaklyadapted} hold for the integer $n-1$.

We start by obtaining a rough upper-bound on the dimension of $S$.

\begin{claim}\label{claim:roughbound}
One has $\dim S \leq \dbinom{n+1}{2}+(n-3)$.
\end{claim}

\begin{proof}
Let first $x \in V \setminus \{0\}$ be arbitrary.
We already find
$$\dim S \leq (n-1)+\dim ((V^\star \otimes x) \cap S)+\dim \overline{S_x},$$
and by induction $\dim \overline{S_x} \leq \dbinom{n}{2}$.
This already yields
$$\dim S \leq \binom{n+1}{2}+(n-1),$$
and in particular $S \neq \End(V)$. It follows that we can choose $x$ so that $V^\star \otimes x \not\subseteq S$
(otherwise by summing all the possibilities when $x$ ranges over $V \setminus \{0\}$, we would recover $S=\End(V)$).
Hence we get the improved inequality $\dim S \leq \binom{n+1}{2}+(n-2)$, and in particular $\dim S \leq n^2-2$,
to the effect that $\dim S^\bot \geq 2$.

We must now improve the upper-bound on $\dim S$ by one extra notch.
To do this, we assume on the contrary that $\dim S = \binom{n+1}{2}+(n-2)$
and examine the consequences of that. So, for every $x \in V \setminus \{0\}$
we learn that $\dim (S \cap (V^\star \otimes x)) \geq n-1$.

By \eqref{eq:dualtransitive} this yields $\dim (S^\bot x) \leq 1$ for all $x \in V \setminus \{0\}$.
Naturally, this invites us to introduce the vector space $\widehat{S^\bot}$ of all the linear mapping
$$\widehat{x} : v \in S^\bot \mapsto v(x),$$
with $x \in V$, and we translate the recent result as saying that the nonzero elements of $\widehat{S^\bot}$ have rank $1$.

Then it is folklore that (at least) one of the following two properties holds:
\begin{enumerate}[(i)]
\item There exists a linear hyperplane $H$ of $S^\bot$ on which all the operators $\widehat{x}$ vanish.
\item There exists a $1$-dimensional linear subspace $D$ of $V$ into which all the operators $\widehat{x}$ map.
\end{enumerate}
Yet as $\dim S^\bot>1$ the first outcome would allow us to find an element $v$ in $H \setminus \{0\}$, and then we would have the absurd
property that $v(x)=\widehat{x}(v)=0$ for all $x \in V$.
Hence there is  a $1$-dimensional linear subspace $D$ of $V$ into which all the operators $\widehat{x}$ map, which means that
$\im v \subseteq D$ for all $v \in S^\bot$.
By double-duality, this shows that $S$ contains all the operators $u \in \End(V)$ that vanish on $D$.
By considering the induced endomorphisms of $V/D$, we would obtain that $\End(V/D)$ is a weakly triangularizable subspace of itself,
and this is absurd because $\dim(V/D) \geq 2$ and $\F$ is not quadratically closed (we could also note that this contradicts the inductive assumption).

This contradiction completes the proof.
\end{proof}

Note that this already yields point (b) of Theorem \ref{theo:weaklyadapted} in the special case $n=3$.

\subsection{Completing the case $n=3$}

Here, we complete the proof of point (a) of Theorem \ref{theo:weaklyadapted} for $n=3$.

So, we assume that $n=3$, we let $(V_i)_{1 \leq i \leq 6}$ be a $3$-complex of $V$, and
we assume that every vector of $V$ that is weakly $S$-adapted belongs to $\underset{i=1}{\overset{6}{\bigcup}} V_i$.

We will consider first the case where $|\F|>4$, and will have to use a different argument in the special case $|\F|=4$.
So, assume first that $|\F|>4$.
As in the previous section of proof, we consider the trace orthogonal space $S^\bot$: we shall prove that every nonzero element of the dual operator space
$\widehat{S^\bot}$ has rank $1$.
To this end, we choose respective basis $(u_1,\dots,u_N)$ and $(e_1,e_2,e_3)$ of $S^\bot$ and $V$, and for
$x \in V$ we denote by $M(x)$ the matrix that represents $\widehat{x}$ in such bases.
Note that $\rk M(x) \leq 1$ for every vector $x$ of $V$ that is not weakly $S$-adapted,
and in particular for every vector of $V \setminus \underset{i=1}{\overset{6}{\bigcup}} V_i$.

Now, let $I \subseteq \lcro 1,3\rcro$ and $J \subseteq \lcro 1,N\rcro$ have cardinality $2$. Consider the function $p_{I,J}$
that maps every $x \in V$ to the minor determinant of $M(x)$ with row indices in $I$ and column indices in $J$.
Then $p_{I,J}$ is a $2$-homogeneous polynomial function, and it vanishes outside of $\underset{i=1}{\overset{6}{\bigcup}} V_i$.
Since $|\F|>4$, the Vanishing Lemma for Homogeneous Polynomials (applied to $m=5$ and $d=2$) yields $p_{I,J}=0$.
Varying $I$ and $J$ then yields that $\rk \widehat{x} \leq 1$ for all $x \in S^\bot$.
Then, we proceed exactly as in the proof of Claim \ref{claim:roughbound} to find a contradiction.

Assume finally that $|\F|=4$. Unfortunately, the previous argument fails, and we must resort to a different one.
We can at least use the Covering Lemma to find a basis $(e_1,e_2,e_3)$ of $V$ made of vectors that are not weakly $S$-adapted
(indeed, otherwise $V \setminus \underset{i=1}{\overset{6}{\bigcup}} V_i$ would be included in a linear hyperplane $V_7$ of $V$,
and the extended list $(V_i)_{1 \leq i \leq 7}$ would then contradict the Covering Lemma).
This yields three $2$-dimensional linear subspaces $W_1,W_2,W_3$ of $V^\star$ such that $W_i \otimes e_i \subseteq S$ for all $i \in \{1,2,3\}$.
We note that the $W_i \otimes e_i$ subspaces are linearly independent, and we deduce that $\dim S \geq 6$. By Theorem \ref{theo:classical},
$S$ is optimal, and since $|\F| \geq 3$ we can use Theorem \ref{theo:grandscardinaux}
to find the structure of $S$: it is represented in some basis by either one of $\Mat_1(\F) \vee \Mat_1(\F)\vee \Mat_1(\F) =\Matt_3(\F)$,
$\mathfrak{sl}_2(\F) \vee \Mat_1(\F)$ or $\Mat_1(\F) \vee \mathfrak{sl}_2(\F)$.
In each case we claim that there exists a linear hyperplane $H$ of $V$ that contains all the non-$S$-adapted vectors,
and this will yield that $(V_1,\dots,V_6,H)$ covers $V$, thereby contradicting the Covering Lemma.

Consider first the case where $S$ is represented by $\Matt_3(\F)$ or
$\mathfrak{sl}_2(\F) \vee \Mat_1(\F)$ in some basis.
Then we have a linear hyperplane $H$ of $V$
that is $S$-invariant. Let $x \in V \setminus H$. Let $\varphi \in V^\star$
be such that $\varphi \otimes x \in S$. For all $y \in H$ we must have $\varphi(y) x \in H$ and hence $\varphi(y)=0$.
It follows that $\varphi=0$ or $\Ker \varphi=H$. Hence $V^\star \otimes x$ is included in $H^\circ \otimes x$, which has dimension $1$.
Therefore $x$ is weakly $S$-adapted. As a consequence, $H$ contains all the vectors of $V$ that are not weakly-$S$-adapted.

Consider finally the case where $S$ is represented by $\Mat_1(\F) \vee \mathfrak{sl}_2(\F)$ in some basis.
Then we have a $1$-dimensional linear subspace $D$ of $V$ that is $S$-invariant and such that, for each $u \in S$, the induced
endomorphism $\overline{u}$ of $V/D$ has trace $0$.
Let $x \in V \setminus D$. Let $\varphi \in V^\star$ be such that $\varphi \otimes x \in S$.
Again, we find that $\varphi$ vanishes on $D$. It follows that $\tr(\varphi \otimes x)=\tr(\overline{\varphi \otimes x})=0$, and hence $\varphi(x)=0$.
Therefore $S \cap (V^\star \otimes x)$ is included in the $1$-dimensional space $(D \oplus \F x)^\circ \otimes x$,
and we conclude that $x$ is weakly $S$-adapted. It follows that $D$ contains all the vectors of $V$ that are not weakly-$S$-adapted.

Hence, in any case we have found that the weakly $S$-adapted vectors are contained in some linear hyperplane of $V$, and we conclude to a contradiction,
as explained earlier. This finishes the proof of point (a) of Theorem \ref{theo:weaklyadapted} in the special case $n=3$.

\subsection{Completing the inductive step}\label{section:inductiveweaklyadapted}

Here, we assume that $n \geq 4$, and we shall prove point (a) thanks to the induction hypothesis.

We perform a \emph{reductio ad absurdum}:
we take a $3$-complex $(V_1,\dots,V_{2n})$ of $V$, and we assume that every weakly $S$-adapted vector
belongs to the union $V_1 \cup \cdots \cup V_{2n}$.

To start with, since $|\F| \geq 4$ we get from the Covering Lemma (applied with $p=3$) that no linear hyperplane $H$ of $V^\star$ satisfies $H \cup \underset{i=1}{\overset{2n}{\bigcup}} V_i^o=V^\star$.
This shows that there exists a basis $(f_1,\dots,f_n)$ of $V^\star$ in which all the vectors are outside of $\underset{i=1}{\overset{2n}{\bigcup}} V_i^o$.
In other words, no $f_i$ vanishes on some $V_k$.

Now, we fix an $i \in \lcro 1,n\rcro$ and consider the linear hyperplane $H:=\Ker f_i$.
We consider the space $S \cap \Hom(U,H)$ and the induced subspace
$$\calT:=\{u_H \mid u \in S \cap \Hom(U,H)\} \subseteq \End(H).$$
Then $\calT$ is a weakly triangularizable linear subspace of $\End(H)$.

Moreover the intersection $(S \cap \Hom(U,H)) \cap (f_i \otimes V)$ equals
$f_i \otimes W_i$ for some linear subspace $W_i$ of $H$.

We shall observe that $\dim W_i>\lfloor \frac{2n}{3}\rfloor$. Assume on the contrary that
$\dim W_i \leq \lfloor \frac{2n}{3}\rfloor=1+\lfloor \frac{2n-3}{3}\rfloor$,
and let us embed $W_i$ inside a linear subspace $G$ of $H$ with dimension $1+\lfloor \frac{2n-3}{3}\rfloor$
(this is possible because obviously $\left\lfloor \frac{2n}{3}\right\rfloor \leq n-1$).
Since none of the $V_j$'s is included in the kernel of $f$, we have
$\dim (V_j \cap H)=\dim(V_j)-1$ for all $j \in \lcro 1,2n\rcro$, and it follows that
$(V_4 \cap H,\dots,V_{2n} \cap H,G)$ is a $3$-complex of $H$, whereas $V_j \cap H=\{0\}$ for all $j \in \{1,2,3\}$.

By induction, there exists a vector $x \in H \setminus \{0\}$ that is weakly $\calT$-adapted but belongs to none of
$V_4 \cap H,\dots,V_{2n} \cap H,G$. We deduce that $x$ does not belong to $V_1 \cup \cdots \cup V_{2n}$
(remember that $V_i \cap H=\{0\}$ for all $i \in \{1,2,3\}$), and hence
$x$ is not weakly $S$-adapted. This yields a $2$-dimensional linear subspace $P$ of $V^\star$ such that
$P \otimes x \subseteq S$. Note that $P \otimes x \subseteq \End(V,H)$, to the effect that we can consider the linear mapping
$$\Phi : \varphi \in P \mapsto (\varphi \otimes x)_{H} \in \End(H).$$
Since $x$ is weakly $\calT$-adapted the range of this map has dimension at most $1$, and hence $\Ker \Phi \neq \{0\}$.
This yields $\varphi\in P \setminus \{0\}$ such that $(\varphi \otimes x)_{|H}=0$, to the effect that $\varphi_{|H}=0$.
Hence $\varphi=\lambda f_i$ for some $\lambda \in \F^\times$, and we deduce that $x \in W_i$.
Yet this contradicts the assumption that $x \not\in G$.

It follows that
$$\dim W_i \geq 1+\left\lfloor \frac{2n}{3}\right\rfloor.$$
Now, since $f_1,\dots,f_n$ are linearly independent the subspaces $f_i \otimes W_i$ are linearly independent in $\End(V)$.

Finally, we need the key observation that $\id_V$ does not belong to the sum $(f_1 \otimes W_1) \oplus \cdots \oplus (f_n \otimes W_n)$.
To prove this, assume the contrary and write $\id_V=\sum_{k=1}^n f_k \otimes x_k$ for some $x_1 \in W_1,\dots,x_n \in W_n$.
We take the predual basis $(e_1,\dots,e_n)$ of $(f_1,\dots,f_n)$, and deduce by evaluating at $e_1$ that 
$e_1=x_1$. Yet $f_1(e_1)=1$ whereas $x_1 \in \Ker f_1$.

We conclude that $\F \id_V,f_1 \otimes W_1,\dots,f_n \otimes W_n$ are linearly independent.
Hence their sum $S'$ satisfies
$$\dim S' \geq 1+\sum_{i=1}^n \dim(f_i\otimes W_i) \geq 1+n+n \left\lfloor \frac{2n}{3}\right\rfloor.$$
Besides, $S' \subseteq \F \id_V+S$, and hence $S'$ is weakly triangularizable.
By Claim \ref{claim:roughbound} applied to $S'$, we find
$$\dim S' \leq \dbinom{n+1}{2}+(n-3).$$
Hence, a contradiction will be found as soon as we establish that
$$n \left\lfloor \frac{2n}{3}\right\rfloor \geq \frac{n(n+1)}{2}-3,$$
and to have this inequality it suffices that
$$n \frac{2n-2}{3} \geq \frac{n(n+1)}{2}-3,$$
which is equivalent to $n(n-7) \geq -18$.
Yet, the least value of $s \mapsto s(s-7)$ when $s$ ranges in $\R$ is $-(7/2)^2$, which is obviously greater than $-18$.

Hence a final contradiction is found. We conclude that the latest assumption was wrong, to the effect that
point (a) in Theorem \ref{theo:weaklyadapted} is satisfied by $S$.
In particular, this obviously yields that at least one vector of $V$ is weakly-$S$-adapted,
and finally we use the inductive technique described in Section \ref{section:adaptedvectorsinduction} to conclude that
$\dim S \leq \dbinom{n+1}{2}$.

This completes the inductive proof of Theorem \ref{theo:weaklyadapted}. Theorem \ref{theo:majodim} immediately follows from it.

\section{Adapted vectors and hurdles}

\subsection{The main obstruction for the existence of an adapted vector}

Unfortunately, Theorem \ref{theo:weaklyadapted} is insufficient in the prospect of analyzing of the weakly triangularizable subspaces with greatest possible dimension:
experience shows that what we really need is an adapted vector if we want to decipher such spaces.

However, adapted vectors do not exist in general for weakly triangularizable subspaces when $\F$
is perfect with characteristic $2$. A basic counterexample is the space $\{0_{n-2}\} \vee \mathfrak{sl}_2(\F)$.
Indeed, it is easy to check that for every $x \in \F^n \setminus \{0\}$
such a subspace contains a trace zero matrix with range $\F x$.
Moreover,  $\{0_{n-2}\} \vee \mathfrak{sl}_2(\F)$ is weakly triangularizable because every matrix of $\mathfrak{sl}_2(\F)$
is weakly triangularizable here: indeed, the characteristic polynomial of any matrix of $\mathfrak{sl}_2(\F)$
takes the form $t^2+\alpha$ for some $\alpha \in \F$,
and hence splits because $\F$ is perfect with characteristic $2$.

This counterexample creates enormous difficulties to prove the existence of adapted vectors, as we shall see.

\subsection{Hurdles and the main theorem on adapted vectors}\label{section:hurdles}

\begin{Def}
Assume that $\dim V \geq 2$. A linear subspace $S$ of $\End(V)$ is called
a \textbf{hurdle} when there exists a $2$-dimensional subspace $P$ of $V^\star$ such that $S$
contains all the operators of the form $\varphi \otimes x$ with $x \in V$ and $\varphi \in P$ such that $\varphi(x)=0$.
\end{Def}

Equivalently, $S$ is a hurdle if and only if there exists a basis $\bfB$ of $V$ such that $S$ contains all the operators that are represented in the basis $\bfB$ by a matrix of the form
$$\begin{bmatrix}
[0]_{(n-2) \times (n-2)} & [?]_{(n-2) \times 2} \\
[0]_{2 \times (n-2)} & N
\end{bmatrix} \quad \text{with $N \in \mathfrak{sl}_2(\F)$.}$$

The following notion will also be quite useful:

\begin{Def}
A linear subspace $S$ of $\End(V)$ is called a \textbf{wall} when there exists a linear hyperplane $H$ of $V$ such that $S$ contains all the endomorphisms of $V$ with kernel $H$.
Then we say that $H$ is associated with $S$.
\end{Def}

We can now state the main result of the second part of the article:

\begin{theo}\label{theo:adaptedvectors}
Let $\F$ be a perfect field with characteristic $2$ that is not quadratically closed, and with $|\F| > 2$.
Let $V$ be a vector space over $\F$ with finite dimension $n \geq 3$.
Let $S$ be a weakly triangularizable subspace of $\End(V)$. Then either $S$ is a hurdle or it has an adapted vector.
\end{theo}

In fact, we will prove a more precise result on the existence of adapted vectors, which is essentially the analogue of Theorem \ref{theo:weaklyadapted}
for adapted vectors.

\begin{theo}\label{theo:adaptedvectorsrefined}
Let $\F$ be a perfect field with characteristic $2$ that is not quadratically closed, and with $|\F| > 2$.
Let $V$ be a vector space over $\F$ with finite dimension $n \geq 3$.
Let $S$ be a weakly triangularizable subspace of $\End(V)$. Then either $S$ is a hurdle or
the union of a $2$-complex of $V$ cannot contain all the $S$-adapted vectors.
\end{theo}

\subsection{Strategy of proof}

The strategy of proof of Theorem \ref{theo:adaptedvectorsrefined}
is largely similar to the one of Theorem \ref{theo:weaklyadapted},
but it is far more difficult. The main difficulties are located in the initialization at $n \leq 4$
and in the possibility of encountering hurdles when trying to apply the induction hypothesis.
To deal with the latter problem, we will require technical results which we call the Confinement Lemmas.
Roughly speaking, these lemmas state that if a weakly triangularizable subspace $S$ contains a very large structure which is close to a hurdle but $S$ is not a hurdle,
then the vectors that are not $S$-adapted are confined in the union of a small number of proper linear subspaces. This can be used to prove
that the union of a $2$-complex of $V$ cannot contain all the $S$-adapted vectors.
With the exception of the First Confinement Lemma, which is an easy consequence of a result featured in \cite{dSPtriangularizable}
and which is called the Erasure Lemma there, the confinement lemmas are the most technical part of the article, and we relegate their proofs to the last section. Apart from the Second Confinement Lemma, the details of proof
are substantially different from the ones of the main proof of \cite{dSPfeweigenvalues2}, and the adaptation is nontrivial.

The case $n \geq 4$ will be considerably helped by the validity of Theorem \ref{theo:grandscardinaux},
and will require an appeal to Atkinson and Lloyd's main theorem of primitive spaces of bounded rank matrices, translated as a theorem on intransitive operator spaces
(see \cite{dSPsemilin}).

\section{Existence of adapted vectors: The case $n \in \{3,4\}$}\label{section:n=34}

Throughout the section, we assume that $\F$ is a perfect field with characteristic $2$ that is not quadratically closed, and with $|\F|>2$.

\subsection{First Confinement Lemma}

The following result will be useful to trap the non-$S$-adapted vectors in specific situations:

\begin{lemma}[First Confinement Lemma]
Let $V$ be a vector space with dimension $n \geq 2$, and let $S$ be a weakly triangularizable linear subspace of $\End(V)$.
Assume that $S$ is a wall, with associated hyperplane $H$. Then all the non-$S$-adapted vectors belong to $H$.
\end{lemma}

This lemma is an easy consequence of the following result, which was proved in \cite{dSPtriangularizable} (for any non-quadratically closed field) in matrix (and transposed) form
(see lemma 5.1 there).

\begin{lemma}[Erasure Lemma]
Let $V$ be a vector space with dimension $n \geq 2$, and let $S$ be a weakly triangularizable linear subspace of $\End(V)$.
Assume that $S$ is a wall, with associated hyperplane $H$. Then $H$ is $S$-invariant.
\end{lemma}

\begin{proof}[Proof of the First Confinement Lemma]
By the Erasure Lemma, the subspace $H$ is $S$-invariant.
Let $x \in V \setminus H$. Let $f \in x^\bot$ be such that $f \otimes x \in S$.

Let $y \in H$. Then $(f \otimes x)=f(y) x$ must belong to $H$, and hence $f(y)=0$.
Hence $H \subseteq \Ker f$. Hence either $f=0$ or $\Ker f=H$, but the latter would lead to $x \in H$, which is not true.
Hence $f=0$ and we conclude that $x$ is $S$-adapted.
\end{proof}

\subsection{The intransitivity of $S^\bot$}

In the remainder of Section \ref{section:n=34}, we fix an integer $n \in \{3,4\}$ and we let
$V$ be an $n$-dimensional vector space over $\F$ and $S$ be a weakly triangularizable subspace of $\End(V)$.
We assume that there is a $2$-complex $(V_i)_{1 \leq i \leq n}$ of $V$ that contains all the $S$-adapted vectors.
We seek to give a direct proof that $S$ is a hurdle.

To this end, we proceed as in Section \ref{section:inductionweaklysetup}, by introducing the trace-orthogonal space $S^\bot \subseteq \End(V)$ and the dual operator space
$\widehat{S^\bot}$, consisting of all the operators
$$\widehat{x} : v \in S^\bot \mapsto v(x) \in V$$
with $x \in V$.
Note immediately that $\dim S^\bot=n^2-\dim S \geq \dbinom{n}{2} \geq n$.

Here is our first result.

\begin{claim}
Every operator in $\widehat{S^\bot}$ has rank at most $n-1$.
\end{claim}

\begin{proof}
We consider a basis $(v_1,\dots,v_r)$ of $S^\bot$ and a basis $(e_i)_{1 \leq i \leq n}$ of $V$.
Let us assign to every $x \in V$ the matrix $M(x)$ of $\widehat{x}$ in these bases.
Whenever $x \in V \setminus \{0\}$ is not $S$-adapted we have $S \cap (V^\star \otimes x) \neq \{0\}$ and hence $\dim (S^\bot x) \leq n-1$
(see \eqref{eq:dualtransitive}), to the effect that $\rk M(x) \leq n-1$.
Now, take an arbitrary subset $I$ of $\lcro 1,r\rcro$ with cardinality $n$, and for $x \in V$
denote by $\Delta_I(x)$ the $n$-by-$n$ minor of $M(x)$ obtained by selecting the column indices in $I$.
Hence $\Delta_I(x)=0$ for all $x \in V \setminus \underset{i=1}{\overset{n}{\bigcup}} V_i$, and the mapping $x \in V \mapsto \Delta_I(x)$
is a homogeneous polynomial function of degree $n$. Now, we take $m:=4$ and $d:=n$, and we observe that the assumptions of the
Vanishing Lemma for Homogeneous Polynomials apply to $\Delta_I$: indeed, in any case $(V_i)_{1 \leq i \leq n}$ contains at most $3$
subspaces of dimension $j$ for all $j \in \{1,\dots,n-2\}$, and either $n=3$ and $(V_i)_{1 \leq i \leq n}$ contains exactly one space of dimension $n-1$
(and $m-d=1$), or $n=4$ and it contains none.
It follows that $\Delta_I=0$. Varying $I$ then yields $\rk M(x) \leq n-1$ for all $x \in V$, which is the claimed result.
\end{proof}

\subsection{A quick review of intransitive operator spaces}

Now we recall the following definitions:

\begin{Def}
Given vector spaces $U$ and $V$, a linear subspace $\calT$ of $\Hom(U,V)$ is called
\textbf{intransitive} when $\calT x \neq V$ for all $x \in U$.
In that case, it is called \textbf{primitively intransitive} whenever there is no nontrivial linear subspace $W$ of $V$ for which
$\pi \calT$ is an intransitive subspace of $\Hom(U,V/W)$, where $\pi : V \twoheadrightarrow V/W$ denotes the standard projection.
\end{Def}

We also recall the following result, which is a variation of Atkinson and Lloyd's main result on primitive spaces of bounded rank matrices
\cite{AtkinsonLloydPrim}:

\begin{theo}[Atkinson and Lloyd's theorem on intransitive operator spaces]
Let $U$ and $V$ be finite-dimensional vector spaces, with $n:=\dim V>0$ such that $n \leq |\F|$.
Let $\calT$ be a primitively intransitive linear subspace of $\Hom(U,V)$.
Then $\dim \calT \leq \dbinom{n}{2}$.
\end{theo}

See theorem 1.22 and proposition 4.5 of \cite{dSPsemilin} for this formulation of the theorem, including a characterization of the spaces with dimension close to the maximal one.

\subsection{Closing in on the structure of $S$}

Our next aim is to prove the following result:

\begin{claim}\label{claim:wallorhurdle}
The space $S$ is a wall or a hurdle.
\end{claim}

To start with, we assume that $S^\bot$ is primitively intransitive. Then by Atkinson and Lloyd's theorem we have $\dim S^\bot \leq \dbinom{n}{2}$
and hence $\dim S \geq \dbinom{n+1}{2}$. Hence $S$ is optimal and we can use Theorem \ref{theo:grandscardinaux}.
It follows that in some basis, $S$ is represented by a space of the form $\calA \vee \Mat_1(\F)$ for some linear subspace $\calA$ of $\Mat_{n-1}(\F)$,
or by a space of the form $\calB \vee \mathfrak{sl}_2(\F)$ for some linear subspace $\calB$ of $\Mat_{n-2}(\F)$.
In the first case it is clear that $S$ is a wall (if the said basis is $(e_i)_{1 \leq i \leq n}$, then a corresponding hyperplane is $\Vect(e_i)_{1 \leq i \leq n-1}$);
in the second case $S$ is obviously a hurdle.

Hence, in the remainder of the proof we assume that $S^\bot$ is not primitively intransitive.
Then can take a maximal nontrivial linear subspace $W$ of $V$ such that, for the standard projection
$\pi : V \twoheadrightarrow V/W$, the space $\pi S^\bot$ is intransitive. Set $s:=\dim(V/W)$.
By the maximality of $W$, the space $\pi S^\bot$ is primitively intransitive, and we can apply Atkinson and Lloyd's theorem to it, which yields $\dim (\pi S^\bot) \leq \dbinom{s}{2}$.
Now, set
$$S_W:=\{u \in S : u_{|W}=0\}.$$
By a standard orthogonality formula (see e.g.\ lemma 2.2 of \cite{dSPtriangularizable2}), we have
$$\dim S_W=(\dim V)(\dim (V/W))-\dim (\pi S^\bot)$$
and hence
$$\dim S_W \geq sn-\dbinom{s}{2}=\dbinom{s+1}{2}+s(n-s).$$
Consider also the linear mapping
$$\Phi : u \in S_W \mapsto u^{V/W} \in \End(V/W),$$
where $u^{V/W}$ stands for the endomorphism of $V/W$ induced by $u$.
The range of $\Phi$ is a weakly triangularizable subspace of $\End(V/W)$, so by Theorem \ref{theo:majodim} we have $\rk \Phi \leq \dbinom{s+1}{2}$.
The kernel of $\Phi$ is the intersection of $S$ with the space
$$\calN_W:=\{u \in \End(V) :  \im u \subseteq W \subseteq \Ker u\},$$
and we note that $\dim \calN_W=s(n-s)$.
Hence by the rank theorem
$$\dim S_W \leq s(n-s)+\dbinom{s+1}{2}.$$
Combining the previous inequalities leads to the following results:
\begin{enumerate}[(i)]
\item $\calN_W \subseteq S$;
\item $\im \Phi$ is a weakly triangularizable subspace of $\End(V/W)$ with dimension $\dbinom{s+1}{2}$;
\end{enumerate}

Next, we can now apply Theorem \ref{theo:grandscardinaux} to $\im \Phi$, and we recover that it is a wall
or a hurdle. Finally, we shall combine this observation with point (i) to deduce that $S$ is itself a wall or a hurdle.
Assume that $\im \Phi$ is a wall (respectively, a hurdle)
with corresponding hyperplane $\overline{H}$ (respectively, with a linear subspace $\overline{H}$ of $V$ with codimension $2$
such that $\im \Phi$ contains all the endomorphisms of $V/W$ with trace zero that vanish on $\overline{H}$).
Set $H:=\pi^{-1}(\overline{H})$ and note that $H$ has codimension $1$ (respectively, $2$) in $V$.
Let then $u \in \End(V)$ vanish on $H$ (respectively, vanish on $H$ and have trace zero). Then in particular $u$ vanishes on $W$,
and the induced endomorphism vanishes on $\overline{H}$ (respectively, vanishes on $\overline{H}$ and has trace zero).
There exists $v \in S_W$ such that $v^{V/W}=u^{V/W}$. Then $u-v \in \calN_W$ and hence $u-v \in S$, by point (i).
Summing, we conclude that $u \in S$.
Hence $S$ is a wall (respectively, a hurdle). This completes the proof of Claim \ref{claim:wallorhurdle}.

\subsection{Completing the proof}

We can now easily finish the proof. Assume that $S$ is a wall, and consider a corresponding linear hyperplane $H$ of $V$.
By the First Confinement Lemma, all the vectors of $V$ that are not $S$-adapted must belong to $H$.
Taking $V_{n+1}:=H$, we deduce from the starting assumption that the subspaces $V_1,\dots,V_{n+1}$ cover $V$.
This however contradicts the Covering Lemma, as in the family $(V_1,\dots,V_{n+1})$ all the components are proper subspaces of $V$,
with no more than three spaces for each dimension.

Therefore $S$ is a hurdle, which completes the proof of Theorem \ref{theo:adaptedvectorsrefined}
for $n \in \{3,4\}$.

\section{Existence of adapted vectors: The inductive step}

The present section consists of the main part of the inductive proof of Theorem \ref{theo:adaptedvectorsrefined}.
Relegated to the final section are the two confinement lemmas that are necessary in the last steps of the proof, and whose proofs
are entirely independent of the induction used here.

We have already closed the case $n \leq 4$ in the previous section, and now we take an integer $n \geq 5$ (this is critical)
and a vector space $V$ of dimension $n$. We assume that the underlying field $\F$ is perfect with characteristic $2$, but not quadratically closed, and that $|\F|>2$.
In particular $|\F| \geq 4$. We also assume that Theorem \ref{theo:adaptedvectorsrefined} holds at the integer $n-1$, i.e., for every $(n-1)$-dimensional vector space
and every weakly triangularizable subspace of it. Throughout, we take a weakly triangularizable subspace $S$ of $\End(V)$ and we assume that
there exists a $2$-complex $(V_1,\dots,V_n)$ of $V$ whose union contains all the $S$-adapted vectors.
Our ultimate goal is to prove that $S$ is a hurdle.

\subsection{The starting argument}

Just like in Section \ref{section:inductiveweaklyadapted}, we use the Covering Lemma in the dual vector space $V^\star$
to find a basis $(f_1,\dots,f_n)$ of the dual space $V^\star$ such that $V_j \not\subseteq \Ker f_i$ for all $i,j$ in  $\lcro 1,n\rcro$.
Again, for $i \in \lcro 1,n\rcro$ we set
$$W_i:=\{x \in \Ker f_i : f_i \otimes x \in S\}.$$
Now, assume temporarily that $\dim W_i>\lfloor \frac{n}{2}\rfloor$ for all $i \in \lcro 1,n\rcro$.
As in Section \ref{section:inductiveweaklyadapted}, we deduce that
$$\dim (S+\F \id_V) \geq 1+n+n \left\lfloor \frac{n}{2}\right\rfloor \geq 1+n+\frac{n(n-1)}{2}=1+\dbinom{n+1}{2},$$
thereby contradicting Theorem \ref{theo:majodim}, since $S+\F \id_V$ is weakly triangularizable.
Hence there is an index $i$ such that $\dim W_i \leq  \lfloor \frac{n}{2}\rfloor$, and without loss of generality we will assume from now on that
$$\dim W_1 \leq  \left\lfloor \frac{n}{2}\right\rfloor.$$

Now, as in Section \ref{section:inductiveweaklyadapted} we introduce the space $H:=\Ker f_1$, the subspace $S \cap \Hom(V,H)$
and the induced subspace
$$\calT:=\{u_H \mid u \in S \cap \Hom(V,H)\} \subseteq \End(H).$$
Let us take a linear subspace $G$ of $H$ that includes $W_1$ and has dimension $ \lfloor \frac{n}{2}\rfloor$.
This time around, we note that $\calV'=(V_3 \cap H,\dots,V_n \cap H,G)$ is a $2$-complex of $H$
and that $V_1 \cap H=\{0\}=V_2 \cap H$.

Let us prove that the union of $\calV'$ must contain all the $\calT$-adapted vectors.
So, assume on the contrary that it does not, and choose a $\calT$-adapted vector $x \in H$ that does not belong to the union of $\calV'$.
In particular, as $x \in H$ and $V_1 \cap H=V_2 \cap H=\{0\}$, we observe that $x \not\in V_1 \cup \cdots \cup V_n$,
to the effect that $x$ is not $S$-adapted. This yields $\varphi \in x^\bot \setminus \{0\}$ such that $\varphi \otimes x \in S$.
However $\varphi \otimes x$ has its range included in $H$, so it induces an endomorphism of $H$ that belongs to $\calT$, leading to
$\varphi_{|H} \otimes x \in \calT$. Observing that $\varphi_{|H}(x)=\varphi(x)=0$, we deduce that $\varphi_{|H}=0$
because $x$ is $\calT$-adapted. In turn, this yields $\varphi=\lambda f_1$ for some $\lambda \in \F^\times$,
and hence $f_1 \otimes x \in S$. We conclude that $x \in W_1$, which contradicts the assumption that $x \not\in G$.

Hence the union of $\calV'$ contains all the $\calT$-adapted vectors.
By induction, we deduce that $\calT$ is a hurdle.

\subsection{A deeper analysis of $S$}

The next and crucial part of the proof is almost a word-for-word adaptation of the corresponding proof in \cite{dSPfeweigenvalues2}
(sections 5.8 and 5.9 therein), and it is useless to give all the details.
Instead, we will explain the overall structure and scrutinize the only differences in the proofs.

Recall that $\dim W_1 \leq \left\lfloor \frac{n}{2}\right\rfloor$ and that $\calT \subseteq \End(H)$ is a hurdle.
The latter fact yields a $2$-dimensional linear subspace $P$ of $H^\star$ such that $\calT$ contains all the trace zero endomorphisms of $H$
that vanish on the pre-dual orthogonal
$$G:={}^{\circ} P \subseteq H,$$
which is a linear subspace of $H$ with codimension $2$.

The next step is the following claim:

\begin{claim}\label{claim:tracezero}
All the endomorphisms of $H$ induced by the elements of $S$ that map into $H$
leave $G$ invariant and induce trace zero endomorphisms of $H/G$.
\end{claim}

This result is a consequence of a result featured in \cite{dSPtriangularizable} and called the Special Erasure Lemma
(lemma 4.2 there):

\begin{lemma}[Special Erasure Lemma]\label{lemma:specialerasure}
Let $C \in \Mat_{n-2,2}(\F)$ and $D \in \Mat_{n-2}(\F)$.
Assume that for every $A \in \mathfrak{sl}_2(\F)$ and every $B \in \Mat_{2,n-2}(\F)$, the matrix
$$\begin{bmatrix}
A & B \\
C & D
\end{bmatrix}$$
is triangularizable. Then $C=0$.
\end{lemma}

Here is a consequence of Lemma \ref{lemma:specialerasure} in terms of hurdles:

\begin{lemma}[Invariant Subspace Lemma for Hurdles]\label{lemma:invariantsubspace}
Let $W$ be a vector space with dimension $m \geq 2$, and let
$\calV \subseteq \End(W)$ be a hurdle, and $W'$ be a linear subspace of $W$ with codimension $2$ such that
$\calV$ contains all the trace zero tensors $f \otimes x$ with $f_{|W'}=0$ and $x \in W$.
Assume that $\calV$ is weakly triangularizable.
Then $W'$ is $\calV$-invariant.
\end{lemma}

\begin{proof}[Proof of Lemma \ref{lemma:invariantsubspace}]
Let us choose a basis $(e_1,\dots,e_m)$ of $W$ in which the last $m-2$ vectors span $W'$,
and let us consider the matrix space $\calN$ that represents $\calV$ in that basis.
Then $\calN^T$ is weakly triangularizable and it includes $\mathfrak{sl}_2(\F) \vee \{0_{n-2}\}$.
Let $M \in \calN^T$, which we write
$M=\begin{bmatrix}
A' & B' \\
C' & D'
\end{bmatrix}$. By adding appropriate matrices of $\mathfrak{sl}_2(\F) \vee \{0_{n-2}\}$, we deduce that 
$\begin{bmatrix}
\lambda A'+A & [0]_{2 \times (n-2)} \\
C' & D'
\end{bmatrix}$ is weakly triangularizable for all $\lambda \in \F$ and all $A \in \mathfrak{sl}_2(\F)$, and
by extracting we deduce that $\F A'+\mathfrak{sl}_2(\F)$ is weakly triangularizable.
If $\tr(A') \neq 0$, this leads to having $\Mat_2(\F)$ weakly triangularizable, thereby contradicting the assumption that $\F$ is not quadratically closed.
Hence $\tr(A')=0$, and by adding appropriate matrices of $\mathfrak{sl}_2(\F) \vee \{0_{n-2}\}$
we deduce that $\begin{bmatrix}
A & B \\
C' & D'
\end{bmatrix}$ is triangularizable for all $A \in \mathfrak{sl}_2(\F)$ and all $B \in \Mat_{2,n-2}(\F)$.
Applying Lemma \ref{lemma:specialerasure} we deduce that $C'=0$. By varying $M$ and coming back to $\calN$,
this shows that $W'$ is $\calV$-invariant.
\end{proof}

Let us now prove Claim \ref{claim:tracezero}.

\begin{proof}[Proof of Claim \ref{claim:tracezero}]
Lemma \ref{lemma:invariantsubspace} applies to $\calT$, which yields that $G$ is $\calT$-invariant.
Now, we extend a basis $(e_1,\dots,e_{n-3})$ of $G$ into a basis of $(e_1,\dots,e_{n-1})$ of $H$,
and we denote by $\calM$ the matrix space associated with $\calT$ in that basis.
Let $u \in \calT$, and denote its matrix in that basis by
$M(u)=\begin{bmatrix}
[?]_{(n-3) \times (n-3)} & [?]_{(n-3) \times 2} \\
[0]_{2 \times (n-3)} & A
\end{bmatrix}$ where $A \in \Mat_2(\F)$. Let $N \in \mathfrak{sl}_2(\F)$ and $\lambda \in \F$.
The space $\calM$ contains the matrix $0_{n-3} \oplus N$, and hence its also contains
$\lambda M(u)+(0_{n-3} \oplus N)$. It follows that $\lambda A+N$ is weakly triangularizable.
Then Theorem \ref{theo:majodim} yields that $A \in  \mathfrak{sl}_2(\F)$.
In other words, the endomorphism of $H/G$ induced by $u$ has trace zero.
\end{proof}

Next, we set
$$\calU:=\{u \in S : \; \im u \subseteq H \quad \text{and} \quad G \subseteq \Ker u\},$$
and we deduce from Claim \ref{claim:tracezero} that all the elements of $\calU$ have trace zero.

From this point on, the proof featured in sections 5.8 and 5.9 of \cite{dSPtriangularizable} can be followed word-for-word
because it only requires that $n \geq 5$, that $|\F| \geq 4$, that $\dim W_1 \leq \left\lfloor \frac{n}{2}\right\rfloor$ 
and that all the $S$-adapted vectors belong to the union
$V_1 \cup \cdots \cup V_n$ for some $2$-complex $(V_1,\dots,V_n)$ of $V$.
We can then directly jump to the conclusion, which states that one of the following two outcomes holds.

\begin{itemize}
\item \textbf{Case 1.} There exists a linear subspace $G'$ of $V$ with codimension $2$ such that $G' \not\subseteq H$
and $S$ contains all the trace zero endomorphisms of $V$ that vanish on $G'$ and map into $H$.

\item \textbf{Case 2.} There is a basis of $V$ in which the matrix space that represents $S$ contains all the matrices of the form
$$\begin{bmatrix}
0 & 0 & 0 & [0]_{1 \times (n-3)} \\
a & \lambda & y & [0]_{1 \times (n-3)} \\
b & x & \lambda & [0]_{1 \times (n-3)} \\
c & b & a & [0]_{1 \times (n-3)} \\
[?]_{(n-4) \times 1} &  [?]_{(n-4) \times 1} &  [?]_{(n-4) \times 1} & [0]_{(n-4) \times (n-3)}
\end{bmatrix}.$$
where $a,b,c,x,y,\lambda$ belong to $\F$ (and the question marks represent arbitrary blocks).
\end{itemize}

In Case 1, the Second Confinement Lemma (to be proved in the next section) states that either $S$
is a hurdle or there is a linear hyperplane $H'$ of $V$ such that all the vectors of $V$ that are not $S$-adapted belong to
$G' \cup H \cup H'$; in the second stated outcome, we would obtain $V=V_1 \cup V_2 \cup \cdots \cup V_n \cup G' \cup H \cup H'$,
which would contradict the Covering Lemma applied to $p=3$: indeed, as $n \geq 4$ there is no linear hyperplane of $V$ among the $V_i$'s, and
among $V_1,\dots,V_n,G',H,H'$ there are at most three linear subspaces with codimension $2$. Hence in Case 1 the space $S$ is a hurdle.

Finally, in Case 2 the Third Confinement Lemma (to be proved in the next section) yields two linear hyperplanes $H_1$
and $H_2$ of $V$ such that all the vectors of $V$ that are not $S$-adapted belong to the union $H_1 \cup H_2$.
Again, this would contradict the Covering Lemma.

As a conclusion, $S$ is a hurdle, and our inductive proof is now complete (or at least will be once the Second and Third Confinement Lemma are proven).

\section{The Second and Third Confinement Lemmas}

Here, we will state and prove the two confinement lemmas that were used in the very end of the proof of Theorem \ref{theo:adaptedvectorsrefined}.
Their proofs are of course independent of the induction process used there.
Beforehand, we collect several basic results that will be used in both proofs.

Interestingly, all the results we will prove here are also valid over fields with cardinality $2$.

\subsection{A lemmas on polynomials with degree $3$}

\begin{lemma}\label{lemma:polylemma}
Let $\F$ be a perfect field with characteristic $2$, but not a quadratically closed one, and let $\delta \in \F$ and
$p \in \F^\times$. Then there exists a scalar $q \in \F$
such that the polynomial $t^3+pt+q$ does not split over $\F$ and does not have $\delta$ among its roots.
\end{lemma}

\begin{proof}
Assume first that $\F$ is finite.
The mapping $f : x \in \F \mapsto x^3+px$ is not injective :
indeed because $\F$ is perfect we can choose $x \in \F^\times$ as a square root of $p$,
and note that $f(x)=f(0)$. Since $\F$ is finite, it follows that $f$ is nonsurjective, and taking $q \in \F$ outside of its range
yields that $t^3+pt+q$ is irreducible, and of course it satisfies the conclusion.

Assume next that $\F$ is infinite.
To start with, we take an arbitrary $r \in \F \setminus \{0,1\}$ such that $t^2+t+r$ is irreducible over $\F$ (at this point, the existence of such elements is irrelevant).
We consider the polynomial $(t+1)(t^2+t+r)=t^3+(r+1)t+r$. Since $\F$ is perfect we have a unique $u \in \F^\times$ such that $r+1=pu^2$, and hence
$(t+1)(t^2+t+r)=u^3((t/u)^3+p (t/u)+q)$ for $q:=r u^{-3}$.
Hence $t^3+pt+q$ does not split over $\F$, and it yields a solution unless $\delta$ is a root of it.
The latter holds only if $q=\delta^3+p \delta$, that is $r^2=u^6(\delta^3+p \delta)^2=(r+1)^3 p^{-3} \delta^2(\delta^2+p)^2$.
If $\delta^2=p$ or $\delta =0$ this condition can never hold since $r \neq 0$. Assume now that $\delta^2 \not\in \{0,p\}$.
Then $r^2=(r+1)^3 p^{-3} \delta^2(\delta^2+p)^2$ amounts to $r$ being a root of
a polynomial with degree $3$ whose entries depend only on $\delta$ and $p$, and we denote the set of those roots by $\calZ$.
In order to conclude it suffices to prove that there exists $r \in \F \setminus (\{0,1\} \cup \calZ)$ such that $t^2+t+r$ is irreducible.
In fact, there are infinitely many $r \in \F$ such that $t^2+t+r$ is irreducible. Indeed, the set of all such $r$'s is the complementary subset of the range of
$\calP : x \in \F \mapsto x^2+x$; moreover $\calP$ is an endomorphism of the group $(\F,+)$, and it is nonsurjective because $\F$
is perfect with characteristic $2$ but not quadratically closed. Hence $\im \calP$ is a proper subgroup of $\F$, and as $\F$
is infinite $\F \setminus \im \calP$ is infinite itself. Hence there exists $r \in \F \setminus \im \calP$ outside of the finite set $\{0,1\} \cup \calZ$,
and by taking such a scalar $r$ we obtain the claimed result.
\end{proof}

\subsection{A lemma on $3$-by-$3$ matrices}

The following lemma is critical to our proofs of the Second and Third Confinement Lemma.

\begin{lemma}\label{lemma:3by3}
Let $\F$ be a perfect field with characteristic $2$, but not a quadratically closed one.
Let $A \in \Mat_3(\F)$.
Assume that the sum $A+(N \oplus 0_1)$ is triangularizable for all $N \in \mathfrak{sl}_2(\F)$.
Then $A$ has one of the following forms:
$$A=\begin{bmatrix}
* & * & 0 \\
* & * & 0 \\
* & * & *
\end{bmatrix} \quad \text{or} \quad A=\begin{bmatrix}
* & * & * \\
* & * & * \\
0 & 0 & *
\end{bmatrix}.$$
\end{lemma}

\begin{proof}
First of all, we can always replace $A$ with $A-(\tr A)I_3$, and hence the situation is reduced to the one where $\tr A=0$.
Now we assume that $\tr A=0$. Assume also that, in the last column of $A$, at least one of the first two entries is nonzero.
Then we try to prove that in the last row of $A$ the first two entries are zero. To this end we note that
conjugating $A$ with $P \oplus I_1$ for an arbitrary $P \in \GL_2(\F)$ does not modify the assumptions.
Hence, without loss of generality we can assume that the third column of $A$ is of the form
$\begin{bmatrix}
0 \\
1 \\
?
\end{bmatrix}$, which we will do from now on.

Next, the situation is unchanged if we replace $A$ with $A-(N \oplus 0_1)$ for an arbitrary $N \in \mathfrak{sl}_2(\F)$ (which does not modify the condition on the trace of $A$),
so we can further reduce the situation to the one where
$$A=\begin{bmatrix}
0 & 0 & 0 \\
0 & \delta & 1 \\
\alpha & \beta & \delta
\end{bmatrix} \quad \text{for some $(\alpha,\beta,\delta) \in \F^3$,}$$
where we have of course used the condition that $\tr A=0$ to see that the second and third diagonal entries are equal.

Then, we take arbitrary scalars $\lambda,x,y$ in $\F$ and obtain that the matrix
$$M=\begin{bmatrix}
\lambda & y & 0 \\
x & \delta+\lambda & 1 \\
\alpha & \beta & \delta
\end{bmatrix}$$
is triangularizable over $\F$. Observing that $\tr(M)=0$ and writing the characteristic polynomial of $M$ as $t^3+c_2 t+c_3$, we compute that
$$c_2=xy+\beta+(\lambda+\delta)^2+\lambda \delta \quad \text{and} \quad
c_3=\alpha y+\delta xy+\lambda(\delta^2+\lambda \delta+\beta),$$
which we combine to find
$$c_3-\delta c_2=\alpha y+\delta \beta+\delta^3+\lambda \beta.$$
Note throughout the remainder of the proof that $t^3+c_2 t+c_3$ must split over $\F$.

Assume first that $\delta \neq 0$.
Then Lemma \ref{lemma:polylemma} applied to $p:=1$ yields $q \in \F$ such that $t^3+pt+q$ does not split over $\F$ and
does not have $\delta$ among its roots.
\begin{itemize}
\item If $\beta \neq 0$, we force $y=1$, then we adjust $\lambda$ so that $c_3-\delta c_2=q-\delta p$,
and finally we adjust $x$ so that $c_2=p$. As $t^3+c_2 t+c_3$ splits over $\F$, we deduce that $\beta=0$.
\item Then $c_3-\delta c_2=\alpha y+\delta^3$. If $\alpha \neq 0$ we take $\lambda=0$, we adjust $y$ so that
$q-\delta p=\alpha y+\delta^3$, we deduce that $y \neq 0$ (otherwise $\delta$ would be a root of $t^3+pt+q$),
and then we adjust $x$ so that $c_2=p$. Again, this yields that $t^3+pt+q$ should split. Hence $\alpha=0$.
\end{itemize}
Hence, if $\delta \neq 0$ we get the desired conclusion that $\alpha=\beta=0$.

Now, we assume that $\delta=0$ and $\beta \neq 0$, and we prove that it leads to a contradiction.
Lemma \ref{lemma:polylemma} applied to $p:=\beta$ yields a $q \in \F$ such that $t^3+pt+q$ does not split over $\F$.

If $\alpha \neq 0$ then we force $x=\lambda=0$ and find $c_2=\beta$ and $c_3=\alpha y$, and finally we can adjust $y$ so that $M$ has characteristic polynomial $t^3+pt+q$,
thereby contradicting our assumptions.
Therefore $\alpha=0$. Next, we have $c_3=\lambda \beta$ and $c_2=xy+\beta+\lambda^2$.
We can adjust $\lambda$ so that $c_3=q$, then take $y=1$ and adjust $x$ so that $c_2=p$, which yields another contradiction.

Assume finally that $\delta=0$ and $\beta=0$. Then $c_2=xy+\lambda^2$ and $c_3=\alpha y$.
Then we take an arbitrary pair $(p,q) \in \F^2$ such that $t^3+pt+q$ does not split over $\F$.
If $\alpha \neq 0$ we take $y:=\alpha^{-1} q$, $x:=0$ and $\lambda$ as a square root of $p$, so that $c_2=p$ and $c_3=q$.
Once more, we deduce that $\alpha=0$.

Hence, in any case we have proved that $\alpha=\beta=0$, which completes the proof.
\end{proof}

\subsection{Second Confinement Lemma}

We can now state and prove the second confinement lemma.

\begin{lemma}[Second Confinement Lemma]\label{lemma:confinement2}
Let $\F$ be a field that is perfect with characteristic $2$ but is not quadratically closed.
Let $V$ be an $\F$-vector space with dimension $n \geq 3$.
Let $H$ be a linear hyperplane of $V$, and $G$ be a linear subspace of $V$ with codimension $2$ such that $G \not\subseteq H$.
Let $S$ be a weakly triangularizable subspace of $\End(V)$ that contains all the trace-zero endomorphisms of $V$ that
vanish on $G$ and map into $H$.
Then one of the following conditions holds:
\begin{enumerate}[(i)]
\item $S$ is a hurdle;
\item There exists a linear hyperplane $H'$ of $V$ such that every
vector of $V$ that is not $S$-adapted belongs to the union $G \cup H \cup H'$.
\end{enumerate}
\end{lemma}

\begin{proof}
We set $P:=G^\circ$ (a $2$-dimensional linear subspace of the dual space $V^\star$).
The assumptions yield that $\varphi \otimes x \in S$ for all $x \in H$ and all $\varphi \in P \cap x^\bot$.

We assume that $S$ is not hurdle, and we will construct a linear hyperplane $H'$ of $V$ such that every
vector of $V$ that is not $S$-adapted belongs to the union $G \cup H \cup H'$.
Throughout, we fix a vector $z \in G \setminus H$.

\vskip 2mm
\noindent \textbf{Step 1. The intersection $S \cap (P \otimes z)$ has dimension at most $1$.} \\
This step only uses the assumption that $\varphi \otimes x \in S$ for all $x \in H$ and all $\varphi \in P \cap x^\bot$, combined with the assumption
that $S$ is not a hurdle. If $S \cap (P \otimes z)$ has dimension $2$ then one proceeds exactly as in 
the proof of lemma 5.5 from \cite{dSPfeweigenvalues2} (step 1 there also) to deduce that $S$ is a hurdle, which we have ruled out from the start.

\vskip 2mm
\noindent \textbf{Step 2. Defining $H'$.} \\
Now, either $S \cap (P \otimes z)=\F (\theta \otimes z)$ for some $\theta \in P \setminus \{0\}$,
in which case we set $H':=\Ker \theta$, or $S \cap (P \otimes z)=\{0\}$, and then we set
$H':=H$ and take an arbitrary $\theta \in P \setminus\{0\}$ (and note that $\theta(z)=0$ in any case).

\vskip 2mm
\noindent \textbf{Step 3. A basic property of some non-$S$-adapted vectors.} \\
Let $\psi \in P$ and $x \in V \setminus (G \cup H)$ be nonzero elements
such that $\psi \otimes x \in S$ and $\psi(x)=0$.
Again, exactly as in step 3 of the proof of lemma 5.5 from \cite{dSPfeweigenvalues2}, 
we combine this with the assumption that $\varphi \otimes x \in S$ for all $x \in H$ and all $\varphi \in P \cap x^\bot$
to deduce that $x \in H'$.

Now, the main intermediate aim is to extend the previous property to all trace zero tensors $\psi \otimes x$ that belong to $S$
and in which $x \in V \setminus (G \cup H)$ and $\psi \neq 0$ (without assuming that $\psi \in P \setminus \{0\}$). 
To obtain this efficiently, we will work with the transpose operator space $S^t$,
which is the linear subspace of $\End(V^\star)$ that consists of all the transposed operators
$$u^t : f \in V^\star \mapsto  f \circ u \in V^\star$$
with $u \in S$.
Note that $S^t$ is weakly triangularizable (since $u^t$ has the same characteristic polynomial as $u$).
For $y \in V$, denote by
$$\mathfrak{i}_V(y) : f \in V^\star \mapsto f(y) \in \F$$
the element of the double-dual $V^{\star\star}$ that is naturally associated with $y$.
Recall that $y \in V \mapsto \mathfrak{i}_V(y) \in V^{\star\star}$ is a vector space isomorphism.
Thus $(f \otimes y)^t=\mathfrak{i}_V(y) \otimes f$ for all $y \in V$ and all $f \in V^\star$.

\vskip 2mm
\noindent \textbf{Step 4. Dual translation of the situation.} \\
We will write $U:=V^\star$ for greater convenience, as well as $D:=H^\circ$, and we recall that $P=G^\circ$.
Note that $D$ is a $1$-dimensional linear subspace of $V^\star$ (whereas $P$ is a $2$-dimensional subspace). The assumption $G \not\subseteq H$
translates into $D \not\subseteq P$, which has the following consequence:
the mapping $\varphi \in D^\circ \mapsto \varphi_{|P} \in P^\star$ is surjective (for details, see again step 4 in the proof of lemma 5.5 of \cite{dSPfeweigenvalues2}).

Coming back to the weakly triangularizable subspace $S^t$, we see that the
starting assumptions mean that $S^t$ contains $\varphi \otimes y$ for all $y \in P$ and all $\varphi \in D^\circ \cap y^\bot$.

\vskip 2mm
\noindent \textbf{Step 5. Preparing the final step.} \\
Let $y \in V \setminus (G \cup H)$ be non-$S$-adapted.
Setting $\varphi_0:=\mathfrak{i}_V(y)$, this translates into the existence of some vector $x \in U \setminus \{0\}$
such that $\varphi_0(x)=0$ and $\varphi_0 \otimes x \in S^t$, and the assumption on $y$ means that $\varphi_0 \not\in P^\circ \cup D^\circ$.
Our goal in the next steps is to prove that $x \in P$, and then by Step 3 we will deduce that $y\in H'$.

So, we assume that $x \not\in P$ and we show that it leads to a contradiction (this contradiction is obtained in the end of Step 7 below).
Since $\varphi_0 \not\in P^\circ$, we can choose a basis $(e_1,e_2)$ of $P$ such that
$\varphi_0(e_1)=0$ and $\varphi_0(e_2)=1$. Then $\varphi_0 \otimes x$ vanishes at $e_1$ and maps $e_2$ to $x$.

Since $\varphi \in D^\circ \mapsto \varphi_{|P} \in P^\star$ is surjective, we can find
$\psi \in D^\circ$ such that $\psi(e_1)=1$ and $\psi(e_2)=0$, and hence $\psi \otimes e_2 \in S^t$,
and $\psi \otimes e_2$ vanishes at $e_2$ and maps $e_1$ to $e_2$. Note that
$$u:=\varphi_0 \otimes x+\psi \otimes e_2,$$ 
belongs to $S^t$ and maps $e_1$ to $e_2$ and $e_2$ to $x$.
The range of $u$ is included in $\Vect(e_2,x)$, and
hence $\Vect(e_1,e_2,x)$ is invariant under $u$. We deduce that $\Vect(e_1,e_2,x)$ is also invariant under any sum of $u$ with tensors of the form
$f \otimes z$ with $z \in P$ and $f \in D^\circ$ such that $f(z)=0$ (the range of each such tensor is included in $P=\Vect(e_1,e_2)$).

\vskip 2mm
\noindent \textbf{Step 6. Proof that $x \not\in D+P$.} \\
This is a critical step where we will use the assumption that $S$ is weakly triangularizable, and the proof slightly differs from the corresponding one in 
\cite{dSPfeweigenvalues2}.

Assume that $x \in D+P$. Choose $e_3 \in D \setminus \{0\}$.
Since $x \not\in P$, we have $\Vect(e_1,e_2,x)=\Vect(e_1,e_2,e_3)$.
Denote by $A$ the matrix, in the basis $(e_1,e_2,e_3)$, of the endomorphism of $\Vect(e_1,e_2,x)$ induced by $\varphi_0 \otimes x$.
The corresponding matrices of endomorphisms induced by the sums of the tensors of the form
$\theta \otimes z$, with $z \in P$ and $\theta \in D^\circ$ such that $\theta(z)=0$, are exactly the matrices of the form
$N \oplus 0_1$ with $N \in \mathfrak{sl}_2(\F)$, and hence we can use Lemma \ref{lemma:3by3}.
Yet we note that $(\varphi_0 \otimes x)(e_2)=x$ and hence in the last row of $A$ the second coefficient is nonzero.
Hence by Lemma \ref{lemma:3by3} we find $(\varphi_0 \otimes x)(e_3) \in \F e_3$.
Yet $(\varphi_0 \otimes x)(e_3)=\varphi_0(e_3)\,x$ with $\varphi_0(e_3) \neq 0$ (indeed, we have $\varphi_0 \not\in D^\circ$ from the start of Step 5).
Hence $x$ is a (nonzero) scalar multiple of $e_3$, and we end up with $\varphi_0(x) \neq 0$, which contradicts the starting assumption that $\varphi_0(x)=0$.

We conclude that $x \not\in D+P$.

\vskip 2mm
\noindent \textbf{Step 7. The final contradiction.} \\
We deduce from the previous step that $D^\circ \cap P^\circ \not\subseteq x^\bot$, i.e.\ we can find
a linear form $f \in D^\circ \cap P^\circ$ such that $f(x) \neq 0$. As a consequence $S^t$ contains $f \otimes z$ for all $z \in P$.
The endomorphism $u$ (introduced in Step 5) induces an endomorphism of $\Vect(e_1,e_2,x)$ whose matrix in the basis $(e_1,e_2,x)$
is of the form
$$\begin{bmatrix}
0 & 0 & 0 \\
1 & 0 & ? \\
0 & 1 & 0
\end{bmatrix}$$
(we have used the observation that $\tr(u)=0$ to obtain that the lower-right entry equals zero).
By adding the previous tensors to $u$ and by considering the induced endomorphisms of $\Vect(e_1,e_2,x)$, we obtain representing matrices in the basis $(e_1,e_2,x)$ of the form
$$\begin{bmatrix}
0 & 0 & \alpha \\
1 & 0 & \beta \\
0 & 1 & 0
\end{bmatrix}$$
where the parameters $\alpha$ and $\beta$ can be chosen at will. All these matrices must 
be triangularizable over $\F$ because $S^t$ is weakly triangularizable.
It follows that their characteristic polynomials are all split over $\F$, i.e., the polynomial 
$t^3+\beta t+\alpha$ splits over $\F$ for all $(\alpha,\beta)\in \F^2$, in obvious conflict with  
Lemma \ref{lemma:polylemma} (even more simply, we can choose an irreducible monic polynomial $q(t)$ of degree $2$,
and simply consider the polynomial $(t+a)\,q(t)$ where $a:=\tr q(t)$).

This final contradiction yields $x \in P$, and by Step 3 we conclude that $y \in H'$. Hence we have proved that all the $S$-adapted vectors
belong to $G \cup H \cup H'$. This completes our proof.
\end{proof}

\subsection{Third Confinement Lemma}

We finish with the Third Confinement Lemma.

\begin{lemma}[Third Confinement Lemma]\label{lemma:confinement3}
Assume that $\F$ is perfect with characteristic $2$ but is not quadratically closed.
Let $V$ be a vector space with dimension $n \geq 4$,
and $S$ be a weakly triangularizable subspace of $\End(V)$.
Assume that in some basis of $V$, the matrix space that represents $S$ contains all the matrices of the form
$$\begin{bmatrix}
0 & 0 & 0 & [0]_{1 \times (n-3)} \\
a & \lambda & y & [0]_{1 \times (n-3)} \\
b & x & \lambda & [0]_{1 \times (n-3)} \\
c & b & a & [0]_{1 \times (n-3)} \\
[?]_{(n-4) \times 1} &  [?]_{(n-4) \times 1} &  [?]_{(n-4) \times 1} & [0]_{(n-4) \times (n-3)}
\end{bmatrix} \quad \text{with $(a,b,c,\lambda,x,y)\in \F^6$.}$$
Then either $S$ is a hurdle or there exist two linear hyperplanes $H$ and $H'$ of $V$
such that every vector of $V$ that is not $S$-adapted belongs to $H \cup H'$.
\end{lemma}

\begin{proof}
We will say that a basis $(e_1,\dots,e_n)$ of $V$ is \textbf{fine} for $S$ when there exists a scalar $\eta \in \F^\times$
such that the matrix space that represents $S$ in this basis contains all the matrices of the form
$$\begin{bmatrix}
0 & 0 & 0 & [0]_{1 \times (n-3)} \\
\eta a & \lambda & y & [0]_{1 \times (n-3)} \\
\eta b & x & \lambda & [0]_{1 \times (n-3)} \\
c & b & a & [0]_{1 \times (n-3)} \\
[?]_{(n-4) \times 1} &  [?]_{(n-4) \times 1} &  [?]_{(n-4) \times 1} & [0]_{(n-4) \times (n-3)}
\end{bmatrix} \quad \text{with $(a,b,c,\lambda,x,y) \in \F^6$.}$$
By assumption such a basis exists, and a key observation is that for any such basis $(e_1,\dots,e_n)$,
if we replace $e_1$ with $\alpha e_1$ for an arbitrary $\alpha \in \F^\times$, and replace $(e_2,e_3)$ with an arbitrary basis of
$\Vect(e_2,e_3)$, then the resulting basis remains fine for $S$.
This is obvious for the problem of modifying $e_1$, and for the one of modifying $(e_2,e_3)$ this is deduced from the following observation:
take $P \in \GL_2(\F)$, set $K:=\begin{bmatrix}
0 & 1 \\
1 & 0
\end{bmatrix}$ and note that $PKP^T=\eta' K$ for some $\eta' \in \F^\times$ because, $K$ being alternating, $PKP^T$
is also alternating.

Note that for all $C \in \F^2$ and all $\eta \in \F^\times$, conjugating the matrix
$\begin{bmatrix}
0 & [0]_{1 \times 2} & 0 \\
\eta KC & [0]_{2 \times 2} & [0]_{2 \times 1} \\
0 & C^T & 0
\end{bmatrix}$ with $I_1 \oplus P \oplus I_1$
yields
\begin{align*}
\begin{bmatrix}
0 & [0]_{1 \times 2} & 0 \\
\eta PKC & [0]_{2 \times 2} & [0]_{2 \times 1} \\
0 & C^TP^{-1} & 0
\end{bmatrix}
& =\begin{bmatrix}
0 & [0]_{1 \times 2} & 0 \\
\eta PKP^T (P^T)^{-1} C & [0]_{2 \times 2} & [0]_{2 \times 1} \\
0 & ((P^T)^{-1}C)^T & 0
\end{bmatrix} \\
& =
\begin{bmatrix}
0 & [0]_{1 \times 2} & 0 \\
\eta\eta' K ((P^T)^{-1} C) & [0]_{2 \times 2} & [0]_{2 \times 1} \\
0 & ((P^T)^{-1}C)^T & 0
\end{bmatrix}.
\end{align*}

Hence, by varying $C$ we obtain, after such a conjugation, all the matrices of the form
$\begin{bmatrix}
0 & [0]_{1 \times 2} & 0 \\
\eta\eta' K C' & [0]_{2 \times 2} & [0]_{2 \times 1} \\
0 & (C')^T & 0
\end{bmatrix}$ with $C' \in \F^2$.

We move on to the first part of the proof. For this part it is much more convenient to work with the transposed space $S^t$, which is weakly triangularizable:
we refer to the proof of the Second Confinement Lemma for the notation.

Now, we take a basis $(e_1,\dots,e_n)$ that is fine for $S$. We set
$$H_1:=\Vect(e_1,e_4,\dots,e_n) \quad \text{and} \quad H_2:=\Vect(e_2,e_3,\dots,e_n),$$
and we denote by $(e_1^\star,\dots,e_n^\star)$ the dual basis of $V^\star$.

Let $y \in V \setminus \{0\}$ be non-adapted to $S$.
We assume that $y \not\in H_1 \cup H_2$, and we shall prove that it leads to a contradiction.

We set $U:=V^\star$ and $\varphi:=\mathfrak{i}_V(y) \in U^\star$.
Then we have a nonzero vector $f \in U$ such that $\varphi \otimes f \in S^t$ and $\varphi(f)=0$.

The restriction of $\varphi$ to the subspace $\Vect(e_2^\star,e_3^\star)$ does not vanish, otherwise
$y \in H_1$. We can also perform a change of basis of the type we have explained earlier so as to reduce the situation to the one where
$\varphi(e_2^\star)=0$. We immediately derive that $\varphi(e_3^\star) \neq 0$.
By rescaling $\varphi$ we can actually assume that $\varphi(e_3^\star)=1$.
And finally we note that $\varphi(e_1^\star) \neq 0$ because $y \not\in \Vect(e_2,\dots,e_n)$, and by rescaling $e_1$ we can also assume that
$\varphi(e_1^\star)=1$. This does not change the spaces $H_1$ and $H_2$, and the resulting basis remains fine for $S$.
To sum up:
$$\varphi(e_1^\star)=\varphi(e_3^\star)=1 \quad \text{and} \quad \varphi(e_2^\star)=0.$$
There will be three main steps from where we are.
We will successively prove that $f \in \Vect(e_1^\star,e_2^\star,e_3^\star,e_4^\star)$,
that $f \in \Vect(e_1^\star,e_2^\star,e_3^\star)$, and in the last step we will show that the latter point leads to a contradiction.

\vskip 3mm
\noindent \textbf{Step 1: Proving that $f \in \Vect(e_1^\star,e_2^\star,e_3^\star,e_4^\star)$.} \\
We assume on the contrary that $f \not\in \Vect(e_1^\star,e_2^\star,e_3^\star,e_4^\star)$.

Note that $S$ contains $e_3^\star \otimes e_2$, and hence $\varphi_0 :=\mathfrak{i}_V(e_2)$
is such that $\varphi_0 \otimes e_3^\star \in S^t$ and $\varphi_0(e_2^\star)=1$.
Then we consider the sum  $\varphi_0 \otimes e^\star_3+\varphi \otimes f$ and we note that it leaves $\Vect(e_2^\star,e_3^\star,f)$ invariant and that its matrix in the basis
$(e_2^\star,e_3^\star,f)$ is of the form $\begin{bmatrix}
0 & 0 & 0 \\
1 & 0 & ? \\
0 & 1 & 0
\end{bmatrix}$. Next, since $f \not\in \Vect(e_1^\star,e_2^\star,e_3^\star,e_4^\star)$ we find some
$\theta \in \{e_1^\star,e_2^\star,e_3^\star,e_4^\star\}^\circ$ such that $\theta(f)=1$.
Then $\theta=\mathfrak{i}_V(z)$ for some $z \in \Vect(e_5,\dots,e_n)$, and because $(e_1,\dots,e_n)$ is fine for $S$
we see that $e_2^\star \otimes z$ and $e_3^\star \otimes z$ belong to $S$, to the effect that
$\theta \otimes e_2^\star$ and $\theta \otimes e_3^\star$ belong to $S^t$.
The endomorphisms $\theta \otimes e_2^\star$ and $\theta \otimes e_3^\star$
leave $\Vect(e_2^\star,e_3^\star,f)$ invariant, and the corresponding resulting matrices in the basis $(e_2^\star,e_3^\star,f)$
are the unit matrices $E_{1,3}$ and $E_{2,3}$. Since $S^t$ is weakly triangularizable, by linearly combining the previous matrices we end up with
the fact that all the monic polynomials of $\F[t]$ with degree $3$ and trace zero split over $\F$, contradicting Lemma \ref{lemma:polylemma}.

We conclude that $f \in \Vect(e_1^\star,e_2^\star,e_3^\star,e_4^\star)$.

\vskip 3mm
\noindent \textbf{Step 2: Proving that $f \in \Vect(e_1^\star,e_2^\star,e_3^\star)$.} \\
This is the most difficult part of the proof. Once more, we assume that $f \not\in \Vect(e_1^\star,e_2^\star,e_3^\star)$.
By scaling $f$ we can, and we will, assume that the coefficient of $f$ on $e_4^\star$ equals $1$ in the basis $(e_1^\star,e_2^\star,e_3^\star,e_4^\star)$.
Now, we come back to matrices for greater convenience.
We denote by $\calM$ the matrix space that represents $S$ in the basis $(e_1,\dots,e_n)$ (which is fine for $S$).
We denote by $M$ the matrix that represents $f\otimes y$ in that basis. Note that $M$ has rank $1$ and trace $0$.
It follows from the previous assumptions and facts that
$$M=\begin{bmatrix}
A_0 & [0]_{4 \times (n-4)} \\
C_1 & [0]_{(n-4) \times (n-4)}
\end{bmatrix}$$
where
$A_0=\begin{bmatrix}
\delta & ? & ? & 1 \\
0 & 0 & 0 & 0 \\
\delta & ? & ? & 1 \\
\gamma & \beta & \alpha & ?
\end{bmatrix}$ for some $\delta,\beta,\alpha,\gamma$ in $\F$, and $C_1 \in \Mat_{n-4,4}(\F)$:
in $A_0$ the vanishing of the second row comes from $\varphi(e_2^\star)=0$, the values of the first and third coefficient in the fourth column come from
$\varphi(e_1^\star)=1=\varphi(e_3^\star)$ and $f \in e_4^\star+\Vect(e_1^\star,e_2^\star,e_3^\star)$, and finally the first column is computed by using
$\rk(M) \leq 1$.
Note also that $\tr(A_0)=0$ and $\rk(A_0) \leq 1$.

Now, thanks to our loosened assumption on $\calM$ we know that it contains a matrix of the form
$$M_0=\begin{bmatrix}
A_1 & [0]_{4 \times (n-4)} \\
C_1 & [0]_{(n-4) \times (n-4)}
\end{bmatrix}$$
where $A_1=\begin{bmatrix}
0 & 0 & 0 & 0 \\
0 & 0 & 0 & 0 \\
\delta & 0 & 0 & 0 \\
\gamma & \eta^{-1} \delta & 0 & 0
\end{bmatrix}$
and $\eta \in \F^\times$ is associated with the fine basis $(e_1,\dots,e_n)$.
Then $M':=M-M_0$ belongs to $\calM$ and takes the form
$$M'=\begin{bmatrix}
\delta & [?]_{1 \times 3} & [0]_{1 \times (n-4)} \\
[0]_{3 \times 1} & B_0 & [0]_{3 \times (n-4)} \\
[0]_{(n-4) \times 1} & [0]_{(n-4) \times 3} & [0]_{(n-4) \times (n-4)}
\end{bmatrix} \quad \text{with} \quad
B_0=\begin{bmatrix}
0 & 0 & 0 \\
? & ? & 1 \\
\beta-\eta^{-1}\delta & \alpha & ?
\end{bmatrix}.$$

For all $N \in \mathfrak{sl}_2(\F)$ we add the matrix $0_1 \oplus N \oplus 0_{n-3}$, we deduce that $\calM$ contains a matrix of the form
$$\begin{bmatrix}
\delta & [?]_{1 \times 3} & [0]_{1 \times (n-4)} \\
[0]_{3 \times 1} & B_0+(N \oplus 0_1) & [0]_{3 \times (n-4)} \\
[0]_{(n-4) \times 1} & [0]_{(n-4) \times 3} & [0]_{(n-4) \times (n-4)}
\end{bmatrix},$$
and we infer by extracting diagonal blocks that $B_0+(N \oplus 0_1)$ is triangularizable.
Then Lemma \ref{lemma:3by3} can be applied, and we deduce that
$\beta=\eta^{-1}\delta$ and $\alpha=0$.
We will now discuss whether the third column of $A_0$ is zero or not, and will try to find a contradiction in any case.

\vskip 3mm
\noindent
\textbf{Case 1. The third column of $A_0$ is nonzero.} \\
Recall that $\alpha=0$. Since $A_0$ has rank $1$ all its columns are scalar multiples of the third one and 
we deduce that the last row of $A_0$ is zero.

In particular $\beta=0$ and hence $\delta=\eta \beta=0$. Hence the first column of $A_0$ is zero.
Then $\tr (A_0)$ is the coefficient of $A_0$ at the $(3,3)$-spot, and hence this coefficient is zero.
Once more, since $\rk(A_0) \leq 1$ the third column must be a scalar multiple of the fourth one, and hence it is zero,
a contradiction.

\vskip 3mm
\noindent
\textbf{Case 2. The third column of $A_0$ is zero.} \\
Hence, for some $\varepsilon \in \F$,
$$A_0=\begin{bmatrix}
\delta & \varepsilon & 0 & 1 \\
0 & 0 & 0 & 0 \\
\delta & \varepsilon & 0 & 1 \\
? & \eta^{-1}\delta & 0 & \delta
\end{bmatrix},$$
where we have used $\tr(A_0)=0$ to see that the coefficient at the $(4,4)$-spot is $\delta$
(and also the fact that $\rk(A_0) \leq 1$ to compute the second column).

\vskip 2mm
\noindent
\textbf{Subcase 2.1. The case where $\varepsilon \neq 0$.} \\
Fix $(\lambda,x,x')\in \F^3$.
Then by using once more the assumption on the fine basis $(e_1,\dots,e_n)$, and by adding $\delta I_n$ and an appropriate matrix of $\calM$, we deduce that
$$\begin{bmatrix}
0 & \varepsilon & 0 & 1 \\
0 & 0 & x' & 0 \\
0 & x & 0 & 1 \\
\lambda & 0 & 0 & 0
\end{bmatrix}$$
is triangularizable over $\F$.
Yet we easily compute that its characteristic polynomial equals
$$t^2 (t^2+xx')+\lambda  ((x+\varepsilon)x'+t^2).$$
Hence this polynomial equals $t^4+c_2 t^2+c_4$ where
$$c_2:=xx'+\lambda \quad \text{and} \quad c_4=\lambda  xx'+\varepsilon \lambda x'.$$
This polynomial should then split over $\F$. We will now prove that the parameters
$\lambda,x,x'$ can be adjusted so that it actually has no root in $\F$.

Let us take an arbitrary pair $(p,q) \in \F^2$ such that $t^2+pt+q$ has no root in $\F$.
Let us choose $z \in \F \setminus \{p\}$. Then, we set $\lambda:=p+z$, which is nonzero. Since $\lambda \varepsilon \neq 0$ we can take the unique $x' \in \F$ such that
$\lambda z+\varepsilon \lambda x'=q$. If $x'=0$ then $z$ would be a root of $t^2+pt+q$, which is forbidden.
Then $x' \neq 0$, and finally we can choose $x \in \F$ such that $xx'=z$, so that $c_2=p$ and $c_4=q$.
Then $t^2+c_2t+c_4$ has no root in $\F$, so neither does $t^4+c_2 t^2+c_4$ and we get the desired contradiction.

\vskip 2mm
\noindent
\textbf{Subcase 2.2. The case where $\varepsilon = 0$.} \\
Then, since $A_0$ has rank at most $1$ its second column is zero. In particular $\delta=0$.
This time around, we use the same technique as in the above to find that for all $a$ and $x$ in $\F$ the matrix
$$\begin{bmatrix}
0 & 0 & 0 & 1 \\
\eta a & 0 & 0 & 0 \\
0 & x & 0 & 1 \\
0 & 0 & a & 0
\end{bmatrix}$$
is triangularizable over $\F$.
One computes that its characteristic polynomial equals $t^4+a t^2+\eta a^2 x$.
As before, we choose $(p,q)\in \F^2$ such that $t^2+pt+q$ has no root in $\F$, and we note that $p \neq 0$ because $\F$ is perfect.
Then we take $a:=p$ and then $x:=\eta^{-1}a^{-2} q$ to obtain a contradiction.

\vskip 3mm
\noindent \textbf{Step 3: The final contradiction.} \\
We conclude from the previous step that $f \in \Vect(e_1^\star,e_2^\star,e_3^\star)$. This shows that
the matrix in $\calM$ that represents $f \otimes y$ in the basis $(e_1,\dots,e_n)$ takes the form
$$\begin{bmatrix}
R_0 & [0]_{3 \times (n-3)} \\
[?]_{(n-3) \times 3} & [0]_{(n-3) \times (n-3)}
\end{bmatrix}$$
where $R_0 \in \Mat_3(\F)$ has its first row nonzero (because $y \not\in H_2=\Vect(e_2,\dots,e_n)$).
Note also that $\rk R_0 \leq 1$ and $\tr(R_0)=0$.

Once more, we use the assumption that the basis if fine for $S$ to obtain that
we find a triangularizable matrix whenever we add $R_0$ to a matrix of the form
$$\begin{bmatrix}
0 & [0]_{1 \times 2} \\
[?]_{2 \times 1} & N
\end{bmatrix} \quad \text{with $N \in \mathfrak{sl}_2(\F)$.}$$
By applying Lemma \ref{lemma:3by3} (after conjugating with a suitable permutation matrix) we deduce
that in the first row of $R_0$ the last two entries are zero. Since the first row of $R_0$ is nonzero and $\rk(R_0) \leq 1$ this shows that the
second and third columns are zero; the coefficient of $R_0$ at the $(1,1)$-spot is then $\tr(R_0)$.
Yet $\tr (R_0)=0$, which yields a final contradiction since the first row of $R_0$ is nonzero.

We conclude that, as claimed, all the vectors that are not $S$-adapted belong to the union $H_1 \cup H_2$.
The Third Confinement Lemma is thus proved.
\end{proof}

Hence, the proof of Theorem \ref{theo:adaptedvectorsrefined} is finally complete. 
Theorem \ref{theo:adaptedvectors} then readily follows from it. In the sequel to this article, we will 
show how the latter leads to the full classification of the optimal weakly triangularizable subspaces of $\Mat_n(\F)$
when $\F$ is finite with characteristic $2$ and more than $2$ elements.


\begin{thebibliography}{1}
\bibitem{AtkinsonLloydPrim}
M.D. Atkinson, S. Lloyd,
{Primitive spaces of matrices of bounded rank.}
J. Austral. Math. Soc. (Ser. A)
{\bf 30} (1980) 473--482.

\bibitem{AtkinsonPrim}
M.D. Atkinson,
{Primitive spaces of matrices of bounded rank II.}
J. Austral. Math. Soc. (Ser. A)
{\bf 34} (1983) 306--315.

\bibitem{Gerstenhaber}
M. Gerstenhaber,
{On nilalgebras and linear varieties of nilpotent matrices (I).}
{Amer. J. Math.}
{\bf 80} (1958) 614--622.

\bibitem{dSPinvitquad}
C. de Seguins Pazzis,
{Invitation aux formes quadratiques.}
Calvage \& Mounet,
Paris (2011).

\bibitem{dSPfeweigenvalues}
C. de Seguins Pazzis,
{Spaces of matrices with few eigenvalues.}
Linear Algebra Appl.
{\bf 449} (2014) 210--311.

\bibitem{dSPtriangularizable}
C. de Seguins Pazzis,
{Spaces of triangularizable matrices.}
Acta Sci. Math. (Szeged)
{\bf 91} (2025) 369--399.

\bibitem{dSPsemilin}
C. de Seguins Pazzis,
{Affine subspaces of units in simple algebras.}
2025, preprint, https://arxiv.org/abs/2508.06934

\bibitem{dSPtriangularizable2}
C. de Seguins Pazzis,
{Spaces of triangularizable matrices II: Finite fields with odd characteristic.}
Linear Algebra Appl.
{\bf 736} (2026) 266-283.


\bibitem{dSPfeweigenvalues2}
C. de Seguins Pazzis,
{Spaces of matrices with few eigenvalues II.}
2026, preprint, https://arxiv.org/abs/2605.05849

\end{thebibliography}
\end{document}